\documentclass[letterpaper,11pt]{amsart}
\usepackage{indentfirst} 
\usepackage{amssymb}
\usepackage{mathtools} 
\usepackage{mathabx} 
\usepackage{amsthm}
\usepackage{thmtools}
\usepackage{enumitem} 
\usepackage[colorlinks=true]{hyperref} 
\usepackage[dvipsnames]{xcolor} 
\usepackage{ifthen}
\usepackage{tikz}
\usetikzlibrary{decorations.pathreplacing}
\usepackage{tikz-cd} 
\usepackage[left=3.2cm, right=3.2cm, bottom=3.4cm]{geometry} 
  
\makeatletter

\def\MRbibitem{\@ifnextchar[\my@lbibitem\my@bibitem}
 
\def\mybiblabel#1#2{\@biblabel{{\hyperref{http://www.ams.org/mathscinet-getitem?mr=#1}{}{}{#2}}}}

\def\myhyperanchor#1{\Hy@raisedlink{\hyper@anchorstart{cite.#1}\hyper@anchorend}}

\def\my@lbibitem[#1]#2#3#4\par{%
  \item[\mybiblabel{#2}{#1}\myhyperanchor{#3}\hfill]#4%
  \@ifundefined{ifbackrefparscan}{}{\BR@backref{#3}}%
  \if@filesw{\let\protect\noexpand\immediate
    \write\@auxout{\string\bibcite{#3}{#1}}}\fi\ignorespaces%
}

\def\my@bibitem#1#2#3\par{%
  \refstepcounter\@listctr
  \item[\mybiblabel{#1}{\the\value\@listctr}\myhyperanchor{#2}\hfill]#3%
  \@ifundefined{ifbackrefparscan}{}{\BR@backref{#2}}%
  \if@filesw\immediate\write\@auxout
    {\string\bibcite{#2}{\the\value\@listctr}}\fi\ignorespaces%
}

\makeatother

\DeclareFontFamily{U} {MnSymbolA}{}
\DeclareFontShape{U}{MnSymbolA}{m}{n}{
   <-6> MnSymbolA5
   <6-7> MnSymbolA6
   <7-8> MnSymbolA7
   <8-9> MnSymbolA8
   <9-10> MnSymbolA9
   <10-12> MnSymbolA10
   <12-> MnSymbolA12}{}
\DeclareFontShape{U}{MnSymbolA}{b}{n}{
   <-6> MnSymbolA-Bold5
   <6-7> MnSymbolA-Bold6
   <7-8> MnSymbolA-Bold7
   <8-9> MnSymbolA-Bold8
   <9-10> MnSymbolA-Bold9
   <10-12> MnSymbolA-Bold10
   <12-> MnSymbolA-Bold12}{}
\DeclareSymbolFont{MnSyA} {U} {MnSymbolA}{m}{n}
 \DeclareFontFamily{U} {MnSymbolC}{}
\DeclareFontShape{U}{MnSymbolC}{m}{n}{
  <-6> MnSymbolC5
  <6-7> MnSymbolC6
  <7-8> MnSymbolC7
  <8-9> MnSymbolC8
  <9-10> MnSymbolC9
  <10-12> MnSymbolC10
  <12-> MnSymbolC12}{}
\DeclareFontShape{U}{MnSymbolC}{b}{n}{
  <-6> MnSymbolC-Bold5
  <6-7> MnSymbolC-Bold6
  <7-8> MnSymbolC-Bold7
  <8-9> MnSymbolC-Bold8
  <9-10> MnSymbolC-Bold9
  <10-12> MnSymbolC-Bold10
  <12-> MnSymbolC-Bold12}{}
\DeclareSymbolFont{MnSyC} {U} {MnSymbolC}{m}{n}

\DeclareMathSymbol{\top}{\mathord}{MnSyA}{219} 
\DeclareMathSymbol{\plus}{\mathord}{MnSyC}{20} 

\declaretheorem[numberwithin=section]{theorem}
\declaretheorem[sibling=theorem]{lemma}
\declaretheorem[sibling=theorem]{corollary}
\declaretheorem[sibling=theorem]{proposition}
\declaretheorem[sibling=theorem,style=definition]{definition}

\declaretheorem[sibling=theorem,style=remark]{remark}

\hypersetup{bookmarksdepth = 3} 
\numberwithin{equation}{section}     

\setlist[enumerate,1]{label={\upshape(\alph*)},ref=\alph*}
\setlist[enumerate,2]{label={\upshape(\arabic*)},ref=\arabic*}

\newcommand{\K}{\mathcal{K}}
\newcommand{\M}{\mathcal{M}}
\newcommand{\R}{\mathbb{R}}
\newcommand{\Z}{\mathbb{Z}}
\newcommand{\N}{\mathbb{N}}

\newcommand{\cR}{\mathcal{R}}

\def\L{{\text{Leb}}}

\def\phi{\varphi}
\def\R{{\mathbb R}}

\def\N{{\mathbb N}}
\def\Z{{\mathbb Z}}

\def\B{{\mathcal B}}

\def\Q{{\mathbb Q}}

\def\M{{\mathcal M}}

\def\le{\leqslant}
\def\ge{\geqslant}

\def\M{\mathcal{M}}

\newcommand{\vertiii}[1]{{\left\vert\kern-0.25ex\left\vert\kern-0.25ex\left\vert #1 
    \right\vert\kern-0.25ex\right\vert\kern-0.25ex\right\vert}}
\newcommand{\invertiii}[1]{{\vert\kern-0.25ex\vert\kern-0.25ex\vert #1 
    \vert\kern-0.25ex\vert\kern-0.25ex\vert}}

\begin{document}

\title{Equidistribution and thermodynamics at infinity}

\date{\today}


\begin{thanks}
{G.I.\ was partially supported by Proyecto Fondecyt Regular 1230100. F. R.\ was partially supported by Proyecto Fondecyt Regular 1231257. A.V.\ was partially supported  by Proyecto Fondecyt Regular  1250928.}
\end{thanks}

\author[G.~Iommi]{Godofredo Iommi}
\address{Facultad de Matem\'aticas,
Pontificia Universidad Cat\'olica de Chile (UC), Avenida Vicu\~na Mackenna 4860, Santiago, Chile}
 \email{\href{mailto:godofredo.iommi@gmail.com}{godofredo.iommi@gmail.com}} 
\urladdr{\url{http://http://www.mat.uc.cl/~giommi/}}

 \author[F.~Riquelme]{Felipe Riquelme}  \address{Instituto de Matem\'aticas, Pontificia Universidad Cat\'olica de Valpara\'iso (PUCV), Blanco Viel 596, Valpara\'iso, Chile}
\email{\href{felipe.riquelme@pucv.cl}{felipe.riquelme@pucv.cl}}
\urladdr{\href{https://sites.google.com/view/feliperiquelmeabarca/}{https://sites.google.com/view/feliperiquelmeabarca/}}

 \author[A.~Velozo]{Anibal Velozo}  \address{Facultad de Matem\'aticas,
Pontificia Universidad Cat\'olica de Chile (UC), Avenida Vicu\~na Mackenna 4860, Santiago, Chile}
\email{\href{apvelozo@uc.cl}{apvelozo@uc.cl}}
\urladdr{\href{https://sites.google.com/view/apvelozo}{https://sites.google.com/view/apvelozo}}

\begin{abstract} We prove level-2 large deviation upper bounds for potentials on countable Markov shifts and for suspension semi-flows over countable Markov shifts. For strongly positive recurrent potentials, we establish equidistribution of weighted empirical measures toward the corresponding equilibrium state. We then apply these results to interval maps, obtaining, in particular, equidistribution  and large-deviation estimates for measures supported on boundary points. For the Gauss map, this yields equidistribution results on rational numbers, including a theorem of David and Shapira.

\end{abstract}

\maketitle

\section{Introduction}

In 1916, Weyl \cite{we} proved that if a polynomial $p(x) \in \mathbb{R}[x]$ has at least one irrational non-constant coefficient, then the sequence $(p(n))_{n\ge 1}$ is equidistributed modulo one. Since then, equidistribution theory has developed in several directions and has found applications in many areas. Its connections with dynamical systems and Diophantine approximation, however, remain central to the subject.

A standard way to prove equidistribution results in dynamics is to derive them from large-deviation estimates. In the early 1990s, large-deviation principles for dynamical systems were established by Kifer \cite{kif} and Young \cite{y}; such results can be used, in particular, to obtain equidistribution of periodic orbits. In this article we follow this general strategy for both discrete  and continuous time systems.

More precisely, we consider countable Markov shifts and suspension semi-flows over them, both of which are dynamical systems defined over non-compact phase spaces. In this setting, orbits and invariant measures may escape to infinity, and the compactness properties underlying the classical large-deviation and equidistribution arguments are no longer available. A central part of our approach is therefore to restore compactness in a way that remains compatible with the dynamics: in the discrete-time setting, we compactify the phase space itself, whereas in the continuous-time setting we work directly with a compactification of the space of invariant measures.

The topology of these compactifications plays an essential role. It allows us to keep track of the escape of mass and to analyze thermodynamic quantities at infinity, such as pressure at infinity. Accordingly, we identify suitable classes of test functions adapted to these topologies, both in the discrete   and continuous time settings. Within this framework, one can, for instance, establish a dual variational principle for suspension semi-flows; see Theorem~\ref{thm:dual}.

We apply our large-deviation and equidistribution results for periodic points and periodic orbits to a class of interval maps, obtaining the equidistribution of empirical measures supported on finite orbits of boundary points. A particularly important example is the Gauss map; in this case, boundary points correspond to rational numbers, and our results yield the equidistribution of their finite Gauss orbits; see Section~\ref{sec:interval}.

\subsection{Countable Markov shifts} 
Let $(\Sigma,\sigma)$ be a topologically mixing countable Markov shift, see Section \ref{sec:cms} for precise definitions. For a fixed symbol $a$ in the alphabet of $\Sigma$, set
\[
\text{Per}_a(n)=\{x\in [a] : \sigma^n(x)=x\}.
\]
For $x\in\text{Per}_a(n)$, define the associated periodic measure
\[
\mu_x= \frac{1}{n} \left(\delta_x+\delta_{\sigma(x)}+\ldots+\delta_{\sigma^{n-1}(x)} \right),
\]
where $\delta_z$ is the Dirac measure at ${z}$. Let $\rho$ be the metric on $\Sigma$ introduced in (\ref{eq:metric}). We denote by $\bar\Sigma$ the completion of $\Sigma$ with respect to $\rho$; this is a compact metric space and the shift map extends to the compactification. We consider the compact dynamical system $(\bar\Sigma,\bar\sigma)$ which was studied in \cite{iv2}. Denote by $\text{UC}_{b,\rho}(\Sigma)$ the space of bounded, uniformly $\rho$-continuous functions on $\Sigma$. Each $g\in\text{UC}_{b,\rho}(\Sigma)$ admits a unique continuous extension $\bar g:\bar\Sigma\to\R$.

Write $\M(\bar\sigma)$ for the space of $\bar\sigma$-invariant Borel probability measures on $\bar\Sigma$, and $P(\cdot)$ for the pressure of potentials defined on $\Sigma$. For a continuous potential $\phi:\Sigma\to\R$, define the \emph{rate function} $I_\phi:\M(\bar\sigma)\to\R$ by
\[
I_\phi(\mu)=\sup\left\{\int \bar g\, d\mu - P(\phi+g) + P(\phi) : g\in \text{UC}_{b,\rho}(\Sigma)\right\}.
\]
The function $I_\phi$ is nonnegative and lower semicontinuous. 

\begin{theorem}[Large deviation upper bound]\label{thm_ldp}
Let $(\Sigma,\sigma)$ be  a topologically mixing countable Markov shift, and  $\phi:\Sigma\to\R$ of summable variations and finite pressure. Then, for any closed subset $\mathcal{K}\subset\M(\bar\sigma)$ we have that
\[
\limsup_{n \to \infty} \frac{1}{n} \log \left(\frac{\sum_{x\in \emph{Per}_a(n),\, \mu_x\in \mathcal{K}} e^{S_n\phi(x)}}{\sum_{x\in \emph{Per}_a(n)} e^{S_n\phi(x)}}\right)
\leq  - \inf \left\{I_\phi(\nu) : \nu \in \mathcal{K} \right\}.
\]
\end{theorem}
Define
\[
s_\infty(\phi)=\inf\{t : P(t\phi)<\infty\}.
\]
Throughout the remainder of this section, assume that $\phi$ has summable variations, $\sup\phi<\infty$, $P(\phi)<\infty$ and $s_\infty(\phi)<1$. In Lemma~\ref{lem:l2} we establish that
\[
I_\phi(\mu)=P(\phi)-h_\mu(\sigma)-\int \phi\, d\mu,
\]
for every $\mu\in\M(\sigma)$ such that $\int \phi\, d\mu>-\infty$.

If $\phi$ is \emph{strongly positive recurrent} (SPR), meaning $P_\infty(\phi)<P(\phi)$, where $P_\infty(\cdot)$ denotes the \emph{pressure at infinity} (see Section~\ref{sec:cms} for precise definitions), then $\phi$ admits a unique equilibrium state $\mu_\phi$ (see \cite[Theorem 3]{sa2}, \cite[Theorem 1.1]{bs}, \cite[Theorem 1.4]{v}).  Furthermore, $I_\phi(\mu)=0$ if and only if $\mu=\mu_\phi$ (see Lemma~\ref{lem:=0}). On the other hand, we prove in Lemma \ref{lem:=01} that if $\phi$ does not have an equilibrium state, then $I_\phi(\mu)=0$  if and only if $\mu$ is the atomic measure $\delta_{\overline{\infty}}$. Observe that this measure can be canonically identified with the zero measure in the space of $\sigma-$invariant subprobabilites $\M_{\le 1}(\sigma)$.  We note that if $\K$  does not contain the unique measure $\mu$ such that $I_\phi(\mu)=0$, then $\inf_{\mu\in \K} I_\phi(\mu)>0$. Combining these facts with Theorem \ref{thm_ldp} yield the following equidistribution result.

\begin{theorem}[Equidistribution of periodic measures]\label{thm_equi}
Let $(\Sigma,\sigma)$  be a topologically mixing countable Markov shift and  $\phi:\Sigma\to\R$ of summable variations with $\sup\phi<\infty$, $P(\phi)<\infty$ and $s_\infty(\phi)<1$.  Let $(A_n)_n$ be a sequence $A_n\subset \emph{Per}_a(n)$ satisfying
$$\lim_{n\to\infty}\frac{1}{n}\log \left(\sum_{x\in A_n} e^{S_n\phi(x)}\right)=P(\phi).$$
Define $\mu_n=\frac{1}{\sum_{x\in A_n} e^{S_n\phi(x)}}\sum_{x\in A_n} e^{S_n\phi(x)}\,\mu_x.$
\begin{enumerate}
    \item Suppose that  $P_\infty(\phi)<P(\phi)$. Then, the sequence $(\mu_n)_n$ converges in the weak$^*$ topology to the equilibrium state of $\phi$. 
    \item Suppose that $\phi$ does not have an equilibrium state.  Then, the sequence $(\mu_n)_n$ converges on cylinders to the zero measure.
\end{enumerate}

\end{theorem}

We emphasize that no combinatorial assumptions on the shift, such as the BIP condition (see Section~\ref{sec:thm}) or finite reducibility \cite{mubook}, are imposed, we only assume topologically mixing. Theorem~\ref{thm_ldp} applies in broad generality, both for the shift space and the potential. Theorem~\ref{thm_equi}, on the other hand,  holds for arbitrary shifts under suitable assumptions on the potential.

Related results, obtained under different assumptions and by different methods, were proved by Takahashi \cite{ta1,ta2}; see also \cite{fl}. Takahashi established level-$2$ large-deviation results for countable Markov shifts, first under the BIP condition \cite{ta1}, and later for arbitrary topologically mixing countable Markov shifts by means of inducing \cite{ta2} (cases where there exists an equilibrium state). In contrast, our approach works directly on a compactification of the phase space and does not require the BIP condition or finite reducibility.

\subsection{Suspension flows over countable Markov shifts} Let $(\Sigma,\sigma)$ be a topologically mixing countable Markov shift, and  $\tau:\Sigma\to\R$ a function with summable variations, satisfying $\inf\tau>0$. Denote by $(Y,\Theta)$, with $\Theta=(\theta_t)_{t\ge 0}$, the suspension semi-flow over $(\Sigma,\sigma)$ with roof function $\tau$ (see Section~\ref{sec:sus} for details). In what follows we will always assume $(Y,\Theta)$ to have finite  entropy.

We say that $y \in Y$ is a periodic point if $\theta_L(y) = y$ for some $L>0$. For $a\in\N$, $0<T_1<T_2$ define
$$\text{FPer}_a(T_1,T_2)=\{(x;s)\in[a]\times [T_1,T_2]:\theta_s(x,0)=(x,0)\}.$$
This set can be identified with that  of periodic orbits starting in $[a]\times \{0\}$ with periods in $[T_1,T_2]$. Also consider the set of all periodic orbits of length at most $T$
$$\text{FPer}_a(T)=\{(x;s)\in[a]\times [0,T]:\theta_s(x,0)=(x,0)\}.$$

Let $\M(\Theta)$ denote the space of $\Theta$-invariant probability measures on $Y$, endowed with the weak$^\ast$ topology. To each periodic point $y$, associate the measure $\nu_y\in\M(\Theta)$ defined  by
$$\int f  d\nu_y = \frac{1}{L} \int_0^L f(\theta_t(y))  dt,$$
for any continuous function $f:Y\to\R$, where $\theta_L(y)=y$. If $y = (x,0)\in Y$ is a periodic point, we also denote the corresponding periodic measure by $\nu_x$. 
 
Let $\M_{\le 1}(\Theta)$ be the space of $\Theta$-invariant sub-probability measures, endowed with the cylinder topology  (see Section \ref{sm} for precise definitions). Under our assumptions, $\M_{\le 1}(\Theta)$ is a compact metric space (see \cite[Theorem 1.6]{v}). In Section~\ref{sec:test}, we introduce a space of test functions $\text{C}_0(Y)$ for this topology; for every $f\in\text{C}_0(Y)$, the map $\nu \mapsto \int f \, d\nu$ is continuous with respect to the cylinder topology on $\M_{\le 1}(\Theta)$ (see Proposition~\ref{prop:test}).

Denote by $\mathrm{C}_b(Y)$ the space of bounded and continuous functions on $Y$ and by $P_\Theta(\cdot)$ the pressure on the flow. For $f\in \mathrm{C}_b(Y)$ 
the \emph{rate function} $I_f:\M_{\leq 1}(\Theta) \to [0,\infty)$ is defined by
\begin{equation*}
I_f(\nu)= \sup \left\{ \int g\, d \nu -P_\Theta(f+g)  +P_\Theta(f) : g \in \text{C}_0(Y)	\right\}.
\end{equation*}
The function $I_f$ is non-negative and lower semicontinuous. Define $\Delta_f:\Sigma\to\R$  by $$\Delta_f(x)=\int_0^{\tau(x)}f(\theta_s(x,0))ds.$$  

The next result is a level-2 large deviation upper bound for finite entropy suspension semi-flows with regular roof function bounded away from zero  and bounded potentials. 

\begin{theorem}\label{thm_ldp2} Let $(\Sigma,\sigma)$ be a topologically mixing countable Markov shift and  $\tau:\Sigma\to\R$ a function of summable variations satisfying $\inf\tau>0$. Denote by  $(Y,\Theta)$ the associated suspension semi-flow and assume that it has finite entropy. Let $f\in \mathrm{C}_b(Y)$ be such that  $\Delta_f$ has summable variations. Let $a\in \N$, $d>0$ and $\mathcal{K}$ be a closed subset of $\M_{\le 1}(\Theta)$. Then,
\begin{align*}
\limsup_{T \to \infty} \frac{1}{T} \log  &\left(\frac{\sum_{(x;s)\in \mathrm{FPer}_a(T-d,T), \nu_{x}\in \mathcal{K}}\exp\big(\int_0^{s}f(\theta_t(x,0))dt\big)}{\sum_{(x;s)\in \mathrm{FPer}_a(T-d,T)}\exp\big(\int_0^{s}f(\theta_t(x,0))dt\big)}\right) \leq - \inf \left\{I_f(\nu) : \nu \in \K \right\}.
\end{align*}
If in addition $P_{\Theta}(f)\ge0$ then we can replace $\mathrm{FPer}_a(T-d,T)$ with $\mathrm{FPer}_a(T)$.  
\end{theorem}

It will be convenient to introduce the following class of roof functions. Define  
 $$\Psi=\{\tau \colon \Sigma \to \mathbb{R}: \tau \text{ has summable variations, }\inf \tau > 0, \text{ and } P(-\tau)<\infty\}.$$
 Note that  $(Y,\Theta)$ has finite entropy if and only if there exists $t_0 >0$ such that $P(-t_0\tau)<\infty$. In particular, if $\tau\in\Psi$, then $(Y,\Theta)$ has finite topological entropy, and this is an equivalence up to a rescaling factor. 
 
Under stronger assumptions on the roof function, the rate function can be described more explicitly. We say that a function 
$\tau:\Sigma \to \mathbb{R}$ belongs to the class $\cR$ if it is bounded away from zero, has summable variations, and satisfies 
\[\lim_{k\to\infty}\inf_{x\in[k]}\tau(x)=\infty.\]
  This class of roof functions was considered in \cite{iv}, where it is shown that for the corresponding suspension flow,  $\M_{\le 1}(\Theta)$ is compact with respect to the cylinder topology.  If $\tau \in \Psi\cap \cR$,  the dual variational principle for the suspension semi-flow holds (see Theorem \ref{thm:dual}). Equivalently, for every $\nu\in\M(\Theta)$, we have
  \[ I_f(\nu)=P_\Theta(f)- h_\nu(\Theta)-\int f d\nu.\]
These results apply, in particular, to countable Markov shifts with the BIP property when $\tau\in \Psi$ has finite first variation (see Lemma \ref{lem:bip}).

A potential  $f:Y \to \mathbb{R}$ is said to be \emph{strongly positive recurrent} (SPR) if $P_{\Theta,\infty}(f)<P_{\Theta}(f)$, where $P_{\Theta,\infty}( \cdot)$ denotes the \emph{pressure at infinity} for the flow (see Section \ref{sec:pre_imfty_flow} for precise definitions). If $f$ is SPR and $\Delta_f$ has summable variations, then $f$ admits a unique equilibrium state (see \cite[Theorem 8.1]{v} and \cite[Theorem 3.5]{ijt}). The next result establishes the equidistribution toward the equilibrium state for large sets of weighted periodic orbits.

\begin{theorem} \label{thm_equi2}
 Let $(\Sigma,\sigma)$ be a topologically mixing countable Markov shift  and  $\tau\in\Psi$. Denote by  $(Y,\Theta)$ the associated suspension semi-flow, which has finite entropy. Let  $f\in \mathrm{C}_b(Y)$ be a SPR potential such that $\Delta_f$ has summable variations. Let $a\in \N$, $d>0$. Then,
 \begin{enumerate}
\item[(a)] For $A_T\subset \emph{FPer}_a(T-d,T)$ such that
$$\lim_{T\to\infty}\frac{1}{T}\log\sum_{(x;s)\in A_T}\exp\bigg(\int_0^{s}f(\theta_t(x,0))dt\bigg)=P_\Theta(f),$$
the sequence of measures $$\nu_T=\frac{1}{\sum_{(x;s)\in A_T}\exp\big(\int_0^{s}f(\theta_t(x,0))dt\big)}\sum_{(x;s)\in A_T}\exp\bigg(\int_0^{s}f(\theta_t(x,0))dt\bigg)\nu_x$$
converges in the weak$^\ast$ topology to the unique equilibrium state of $f$ as $T\to\infty$. 
\item[(b)] There are sets of periodic orbits $A'_T\subset \emph{FPer}_a(T-d,T)$ that
$$\lim_{T\to\infty}\frac{\sum_{(x;s)\in A'_T}\exp(\int_0^sf(\theta_t(x,0))dt)}{\sum_{(x;s)\in \mathrm{FPer}_a(T-d,T)}\exp(\int_0^sf(\theta_t(x,0))dt)}=1,$$
such that for every sequence $(x_n;s_n)\in A'_{T_n}$, with $T_n\nearrow\infty$, we have $(\nu_{x_n})_n$ converges in the weak$^\ast$ topology to the equilibrium state of $f$.
\end{enumerate}
If, in addition, $P_{\Theta}(f)\ge0$ then in the above results we can replace $\mathrm{FPer}_a(T-d,T)$ with $\mathrm{FPer}_a(T)$.
\end{theorem}

These  large deviation and equidistribution theorems for suspension semi-flows over countable Markov shifts, generalize to a non-compact setting previous results by Kifer \cite{kif}  and Pollicott \cite{po}.

\subsection{Equidistribution for interval maps} \label{intro:interval}

We apply our large-deviation and equidistribution results for suspension semi-flows over countable Markov shifts to obtain equidistribution results for finite orbits of boundary points of interval maps with countably many full branches accumulating at zero. More precisely, let $(d_n)_n$ be a strictly decreasing sequence of real numbers such that $d_1=1$ and $\lim_{n\to\infty}d_n=0$. Set
$I_n=(d_{n+1},d_n]$. We consider maps $T:\bigcup_{n\geq 1} I_n \to [0,1]$ such that, for each $n \in \N$, the restriction $T|_{I_n}$ is monotone and onto. We assume that $T$ is regular and expanding in the sense of Section~\ref{sec:interval}. This class includes, in particular, the Gauss map. Under these assumptions, the geometric potential $-\log |T'|$ has a unique equilibrium measure, which we denote by $\mu_1$.

Let $\B= \bigcup_{n \geq 1} T^{-n}(0)$ be the collection of boundary points corresponding to the Markov partition.  For $b\in\mathcal B$, let $\ell(b)$ denote the length of its finite orbit, and define the associated empirical measure by
\[
\mu_b=
\frac{1}{\ell(b)}
\sum_{i=0}^{\ell(b)-1}\delta_{T^i b}.
\]
For each $k\in\mathbb N$, set
\[
\mathcal B_k=
\left\{
b\in\mathcal B:
\left|\big(T^{\ell(b)}\big)'(b)\right|<k
\right\}.
\]

The following theorem, obtained as an application of our previous results, establishes the equidistribution of finite orbits of boundary points.

\begin{theorem}\label{main:interval}
Let $T:(0,1]\to [0,1]$ be an EMR interval map with superlinear branch expansion. Then there exist sets $A_k\subset \mathcal B_k$, $k\in\mathbb N$, such that
\[
\lim_{k\to\infty}\frac{\#A_k}{\#\mathcal B_k}=1,
\]
and, for every sequence $(b_k)_k$ with $b_k\in A_k$, the following hold:
\begin{enumerate}
    \item[(1)] the sequence of measures $(\mu_{b_k})_k$ converges in the weak$^\ast$ topology to the equilibrium measure $\mu_1$ of $-\log |T'|$;
    \item[(2)]
    \[
    \lim_{k\to\infty}
    \frac{1}{\ell(b_k)}
    \log \left|\big(T^{\ell(b_k)}\big)'(b_k)\right|
    =
    \int \log |T'|\,d\mu_1.
    \]
\end{enumerate}
\end{theorem}

In the particular case of the Gauss map, the set $\mathcal B$ is precisely the set of rational numbers in $(0,1)$. We therefore recover the following theorem of David and Shapira~\cite[Theorem~1.1(2)]{ds}.

\begin{theorem}[David--Shapira]\label{os}
For every $q\in\mathbb N$, let
\[
B_q=
\left\{
\frac{p}{q}\in(0,1):
1\leq p<q,\ \gcd(p,q)=1
\right\}.
\]
Then there exist subsets $A_q\subset B_q$ such that $\lim_{q\to\infty}\frac{\#A_q}{\#B_q}=1,$ 
and, for every sequence $(x_q)_q$ with $x_q\in A_q$, the sequence of measures $(\mu_{x_q})_q$ converges in the weak$^\ast$ topology to the Gauss measure $\mu_G$. Moreover,
\[
\lim_{q\to\infty}
\frac{\ell(x_q)}{2\log q}
=
\frac{6\log 2}{\pi^2},
\]
where $\ell(x)$ denotes the length of the continued fraction expansion of $x\in\mathbb Q\cap(0,1)$.
\end{theorem}

Our large-deviation results also yield the following polynomial deviation estimate for rational numbers, analogous to \cite[Theorem~1.2(2)]{dtms}. Let
\[
\mathbb{Q}(q)=
\left\{
\frac{p}{r}\in(0,1):
1\leq p<r\leq q,\ \gcd(p,r)=1
\right\},
\qquad
\Phi(q)=\#\Q(q).
\]

\begin{theorem}\label{thm:prob}
For every sufficiently small $\epsilon>0$, there exists $r_\epsilon>0$ such that, for all sufficiently large $q\in\mathbb N$,
\[
\frac{1}{\Phi(q)}
\#
\left\{
\frac{p}{r}\in\Q(q):
\left|
\frac{\ell(p/r)}{2\log r}
-
\frac{6\log 2}{\pi^2}
\right|
\geq \epsilon
\right\}
\leq q^{-2r_\epsilon}.
\]
\end{theorem}

It is known, at least since von Neumann's work in 1925, that any dense sequence in $[0,1]$ can be reordered so as to become uniformly distributed; see \cite[Corollary~4.2]{kn} for a modern reference. This observation shows that equidistribution depends not only on the set of points under consideration, but also on the order in which they are enumerated. This point is particularly relevant in Theorem~\ref{os}, where two natural orderings of rational numbers are compared: ordering by the size of the denominator, and ordering by the length of the orbit under the Gauss map, equivalently by the length of the continued fraction expansion. Our proof relates these two orderings by passing to an appropriate suspension semi-flow.

\noindent\\
\textbf{Acknowledgement.} The second author gratefully acknowledges the Pontificia Universidad Cat\'olica de Chile for its warm hospitality through the Program ``Visitas Nacionales'' of the Faculty of Mathematics during the summers of 2025 and 2026.


\section{Countable Markov shifts}  \label{sec:cms}

In this section, we provide the necessary background on the topologies of the space of invariant measures and on the associated thermodynamic formalism in this setting. Let  $M$ be a $\N \times  \N$ matrix  with entries $0$ or $1$. The symbolic space associated to $M$ is defined by
 \begin{equation*}
 \Sigma=\left\{ (x_1, x_2, \dots) \in \N^{\N}: M(x_i,x_{i+1})=1 \text{ for every } i \in \N \right\}.
\end{equation*} 
We endow $\N$ with the discrete topology and $\N^{\N}$ with the product topology. On $\Sigma$ we consider the induced topology. In general, this is a non-compact space. The \emph{shift map} $\sigma:\Sigma \to \Sigma$ is defined by $\sigma(x)=(x_2,x_3,\ldots)$, where $x=(x_1, x_2, \dots ) \in \Sigma$. The dynamical system $(\Sigma,\sigma)$ is called  a  \emph{countable Markov shift}. 

We say that $a_1\ldots a_{m}$ is an \emph{admissible word} if $M(a_{i},a_{i+1})=1$ for every $1\le i\le m-1$. A \emph{cylinder} of length $m$ is a set of the form
\begin{equation*}
[a_1,\ldots,a_{m}]= \left\{ x=(x_1,x_2,\ldots)\in \Sigma :  x_i=a_i  \text{ for } 1 \le i \le m \right\}.
\end{equation*} 
Note that  $[a_1,\ldots,a_{m}] \neq \emptyset$ if and only if  $a_1a_2\ldots a_m$ is an admissible word. The collection of cylinder sets forms a basis for the topology on $\Sigma$. 

Let $d$ be the metric on $\Sigma$ defined by $d(x,y)=0$ if $x=y$, and for $x\ne y$ we set $d(x,y)=\frac{1}{2^n}$, where $n=\min\{i\in\N:x_i\ne y_i\}$. In this metric balls are cylinders. Moreover, it generates the topology.

We say that $(\Sigma,\sigma)$ is \emph{topologically transitive}  if for every $a, b \in \N$, there exists an admissible word that starts with $a$ and ends with $b$. We say that $(\Sigma,\sigma)$ is \emph{topologically mixing} if for every pair $a, b \in \N$, there exists a number $N(a,b)$ such that for every $n \ge N(a,b)$, there exists an admissible word of length $n$ that starts with $a$ and ends with $b$. 

The \emph{$n$-th variation} of a function $\phi:\Sigma \to \R$ is defined by  
$$V_{n}(\phi):= \sup \{\,| \phi(x)-  \phi(y)| : x,y \in \Sigma,\; x_{i}=y_{i}\ \text{for}\ 1 \leq i \leq n \}.$$
We say that $\phi$ has \emph{summable variations} if $\sum_{n=2}^{\infty} V_n(\phi)< \infty$. We say that $\phi$ is \emph{locally H\"older} if there exist constants $\theta \in (0,1)$ and $C>0$ such that $V_{n}(\phi) \leq C \theta^{n}$ for all $n \geq 2$.  Denote by $\text{C}_b(\Sigma)$  the space of bounded, real-valued, continuous functions on $\Sigma$. The $C^0$-norm on $\text{C}_b(\Sigma)$ is defined by  $\|f\|_\infty=\sup_{x\in\Sigma}|f(x)|$. Let $\mathrm{UC}_d(\Sigma)$ be the space of uniformly continuous functions with respect to $d$. Observe that $\phi\in \mathrm{UC}_d(\Sigma)$ if and only if $\lim_{n\to\infty}V_n(\phi)=0$. 

Denote by $\M(\sigma)$ the space of $\sigma$-invariant Borel probability measures on $\Sigma$ and by $\M_{\le1}(\sigma)$ the space of $\sigma$-invariant Borel sub-probability measures on $\Sigma$. Let $(\mu_n)_n$, $\mu$ be measures in $\M_{\le1}(\sigma)$.  We say that $(\mu_n)_n$ converges in the \emph{weak$^\ast$ topology} to $\mu$ if $\lim_{n\to\infty}\int f d\mu_n=\int fd\mu$, for every $f\in \text{C}_b(\Sigma)$ (see \cite{par} for properties).  We say that $(\mu_n)_n$ converges \emph{on cylinders}, or in the \emph{cylinder topology}, to $\mu$ if $\lim_{n\to\infty}\mu_n(C)= \mu(C)$, for every cylinder $C\subset\Sigma$ (see \cite{iv} for properties).

\subsection{Compactification and test functions for the cylinder topology}\label{sec:test} 

The purpose of this section is twofold. First, we describe a compactification of $\Sigma$, studied in \cite{gs,it,iv2,z}, to which $\sigma$ extends continuously, ensuring a compact set of invariant measures. Second, we use this compactification to introduce a manageable class of test functions for the cylinder topology. On $\Sigma$ we introduce the metric 
\begin{align}
\label{eq:metric}\rho(x,y)=\sum_{n=1}^\infty \frac{1}{2^n}\left|\frac{1}{x_n}-\frac{1}{y_n}\right|,
\end{align}
where $x=(x_1,x_2,\ldots)$, $y=(y_1,y_2, \ldots)$.  This metric is totally bounded, and its completion produces a compact space 
$\bar\Sigma$, which can be seen as a subset of $(\mathbb{N}\cup\{\infty\})^\mathbb{N}$ 
where entries are allowed to take the value $\infty$. Because the shift map 
$\sigma:\Sigma\to\Sigma$ is uniformly continuous with respect to $\rho$, it 
extends naturally to a continuous map $\bar\sigma:\bar\Sigma\to\bar\Sigma$. In this 
way, $(\bar\Sigma,\bar\sigma)$ provides a compactification of the original system 
$(\Sigma,\sigma)$. Denote by $\M(\bar\sigma)$ the space of $\bar\sigma$-invariant Borel probability measures on $\bar\Sigma$, equipped with the weak$^\ast$ topology.

Let $\text{UC}_{\rho}(\Sigma)$ denote the space of functions on $\Sigma$ that are uniformly continuous with respect to $\rho$, and let $\text{UC}_{b,\rho}(\Sigma)$ be the subspace of bounded ones. Continuous extension and restriction establish a bijection between $\text{UC}_{b,\rho}(\Sigma)$ and $\text{C}(\bar\Sigma)$. For $f\in\text{UC}_{b,\rho}(\Sigma)$, we denote its extension to $\bar\Sigma$ by $\bar f\in\text{C}(\bar\Sigma)$

Denote by $\bar\infty=(\infty,\infty,\ldots)$. It is shown in \cite[Corollary 3.7]{iv2} that $\bar\infty \notin \bar\Sigma$ if and only if $(\Sigma, \sigma)$ admits  a finite uniform Rome \footnote{Let $G = (V, E)$ be the directed graph associated with a countable Markov shift $(\Sigma, \sigma)$. A subset $F \subseteq V$ is called a \emph{uniform Rome} if there exists $N \in \mathbb{N}$ such that $V \setminus F$ contains no paths in $G$ of length greater than $N$. A \emph{finite uniform Rome} is a uniform Rome where $F$ is finite.}, a combinatorial condition that prevents escape of mass for invariant measures. If $\bar\infty\in \bar\Sigma,$ we define 
\begin{align}\label{eq:test1}
\text{UC}_{b,\rho}^0(\Sigma)=\left\{f\in \text{UC}_{b,\rho}(\Sigma): \bar{f}(\bar \infty)=0 \right\},
\end{align}
and, otherwise, we set $\text{UC}_{b,\rho}^0(\Sigma)= \text{UC}_{b,\rho}(\Sigma)$. The set  $\text{UC}_{b,\rho}^0(\Sigma)$ is a space of test functions for the cylinder topology. 

\begin{lemma}\label{lem:testsigma}
 Let $(\Sigma,\sigma)$ be a transitive countable Markov shift. Let $\lambda\in[0,1]$ and $(\mu_n)_n,\mu$  measures in $\M(\sigma)$. Then, $(\mu_n)$ converges on cylinders to $\lambda\mu$ if and  only if 
 $$\lim_{n\to\infty}\int f \, d\mu_n=\int f \, d(\lambda\mu),$$
for every $f\in \emph{UC}_{b,\rho}^0(\Sigma)$. 
\end{lemma}

\begin{proof} First, consider the case where $\bar\infty \notin \bar\Sigma$. By \cite[Corollary 3.7]{iv2}, we have that $\lambda=1$. Since $ \text{UC}_{b,\rho}(\Sigma)$ is a convergence defining class for the weak$^\ast$ topology (see \cite[Remark 3.1]{iv2}), the result follows.

Now assume that $\bar\infty\in \bar\Sigma$. By \cite[Theorem 3.5]{iv2}, a sequence $(\mu_n)_n$ converges on cylinders to $\lambda\mu$ if and only if it converges in the weak$^\ast$ topology of $\M(\bar\Sigma)$ to $\lambda\mu+(1-\lambda)\delta_{\bar\infty}$. Here $\delta_{\bar\infty}$ denotes the Dirac measure at $\bar\infty$. Equivalently,
$$\lim_{n\to\infty} \int g \, d\mu_n=\lambda\int g \, d\mu+(1-\lambda)g(\bar\infty),$$ 
for every $g\in \text{UC}_{b,\rho}(\Sigma)$. This in turn is equivalent to 
$$\lim_{n\to\infty} \int f \, d\mu_n=\lambda\int f \, d\mu,$$
 for every $f\in \text{UC}^0_{b,\rho}(\Sigma)$, upon writing $f=g-g(\bar\infty)$. 
\end{proof}

\begin{remark}In \cite[Lemma 3.19]{iv}, we described the space of test functions for the cylinder topology as the $\text{C}^0$-closure of the set of functions that can be written as finite linear combinations of characteristic functions of cylinder sets. In the present setting, we restrict to invariant measures, which allows us to work with a broader class of test functions.\end{remark}

\begin{remark} \label{rem:weak=cyl}  When there is no escape of mass the cylinder and the weak$^*$ topologies coincide. Indeed, if $\mu \in \mathcal{M}(\sigma)$ and $(\mu_n)_n$ is a sequence of measures in $\mathcal{M}(\sigma)$, then $(\mu_n)_n$ converges to $\mu$ in the weak$^*$ topology if and only if it converges to $\mu$ on cylinders, see \cite[Lemma 3.17]{iv}.
\end{remark}

We conclude this subsection with a few observations.

\begin{lemma}\label{lem:lip}
Let $f : \Sigma \to \R$ be a Lipschitz function with respect to $\rho$. Then $f$ has summable variations and $V_1(f)<\infty$.
\end{lemma}

\begin{proof}
Assume that there exists $C>0$ such that for every $x,y\in\Sigma$ we have $|f(x)-f(y)|\le C\rho(x,y)$.  Note that if $x, y \in \Sigma$ are such that $x_i = y_i$ for all $i \in \{1, \ldots, n\}$, then $\rho(x, y) \le \frac{1}{2^n}$. Therefore, $|f(x)-f(y)|\le C\rho(x,y)\le \frac{C}{2^n}$, and thus $V_n(f)\le \frac{C}{2^n}$. 
\end{proof}

\begin{lemma}\label{lem:appr} If $f\in \emph{UC}_{b,\rho}(\Sigma)$ then it can be $\emph{C}^0-$approximated by potentials with summable variations. Moreover, $V_1(f)<\infty$. 
\end{lemma}

\begin{proof} Let $\bar{f}$ be the continuous extension of $f$ to $\bar{\Sigma}$. It is well known that Lipschitz functions are dense in the space of continuous functions on a compact metric space with respect to the $\text{C}^0-$norm. It follows that $\bar{f}$ can be approximated by Lipschitz functions with respect to $\rho$. The result now follows from Lemma \ref{lem:lip}.
\end{proof}

\subsection{Thermodynamic formalism} \label{sec:thm}

Let $\phi:\Sigma\to\R$ be a potential with summable variations. The \emph{Gurevich pressure} of $\phi$, introduced by Sarig in \cite{sa1}, is defined as
\begin{align*}\label{eq:press1}
P(\phi) = \limsup_{n \rightarrow \infty} \frac{1}{n} \log \sum_{\sigma^{n}x=x}  \exp \left(\sum_{k=0}^{n-1} \phi(\sigma^{k}x)\right) \chi_{[i]}(x),  
\end{align*}
where $ \chi_{[i]}(x)$ is the indicator of the cylinder $[i]$. If the system is transitive, $P(\phi)$ is independent of the choice of $[i]$; if it is topologically mixing, the limsup is in fact a limit. 

The Gurevich pressure  is convex  and satisfies the following approximation property:
$$P(\phi) = \sup \left\{ P(\phi|K) : K \textrm{ is compact and } \sigma\textrm{-invariant}\right\},$$ 
where $P(\phi| K)$ is the topological pressure of $\phi$ restricted to the compact set $K$ (for definition and properties see \cite[Chapter 9]{w}). Moreover, it satisfies the Variational Principle (see \cite{sa1, ijt}):
$$ P(\phi)= \sup \left\{ h_\mu(\sigma) + \int \phi \ d \mu : \mu \in \M(\sigma) \text{ and } - \int \phi \ d \mu < \infty \right\},  $$
where  $h_\mu(\sigma)$ denotes the entropy of the measure $\mu$ (for a precise definition, see \cite[Chapter 4]{w}). A measure $\mu \in \M(\sigma)$ attaining the supremum  is called an \emph{equilibrium measure} or \emph{equilibrium state}  for $\phi$.

A countable Markov shift $(\Sigma, \sigma)$ is said to satisfy the \emph{Big Images and Pre-images property} (BIP property)  if there exists a finite set $B \subset \N$ such that, for every $a \in \N$, there exist $b,b' \in B$ with  $bab'$  admissible. For example, the full shift $(\N^\N,\sigma)$  satisfies the BIP property. Under this combinatorial assumption, the characteristic function in the definition of the Gurevich pressure can be omitted \cite[Corollary 1]{sa3}. That is, if $(\Sigma, \sigma)$ satisfies the BIP property and $\phi$ has summable variations, then 
\begin{equation*}
P(\phi) = \lim_{n \rightarrow \infty} \frac{1}{n} \log \sum_{\sigma^{n}x=x}  \exp \left(\sum_{k=0}^{n-1} \phi(\sigma^{k}x)\right).\end{equation*}

For our applications we will require thermodynamic quantities for potentials with less regularity.

\begin{lemma}\label{lem:c0} Let $(\Sigma,\sigma)$ be a topologically mixing countable Markov shift and $\phi:\Sigma\to\R$  a potential with summable variations and finite pressure. Let $g\in \mathrm{UC}_{b,\rho}(\Sigma)$. Then, 
\begin{align*}
\sup \left\{	h_{\mu}(\sigma)  + \int (\phi +g) \, d\mu : \mu \in \M(\sigma) \text{ and }\int (\phi +g) \, d\mu  <\infty		\right\} = &\\ \lim_{n \to \infty} \frac{1}{n} \log \sum_{\sigma^nx=x}\exp \left(\sum_{k=0}^{n-1} (\phi+g) (\sigma^{k}x)\right) &\chi_{[a]}(x). 
\end{align*}
\end{lemma}

\begin{proof} Since both the Gurevich pressure and the variational formula depend continuously on the potential with respect to the $\text{C}^0$-norm, it suffices to approximate $\phi+g$ in $\text{C}^0$ by potentials with summable variations, for which the Gurevich pressure is given by a limit and the variational principle holds. This follows directly from Lemma \ref{lem:appr}.
\end{proof}

In \cite[Theorem 1.4 and Remark 6.1]{iv2},  we established a version of the dual variational principle for transitive countable Markov shift. Specifically, if $\phi\in\mathrm{UC}_d(\Sigma)$ satisfies $s_\infty(\phi)<1$, $\sup \phi <\infty$, $P(\phi)<\infty$ and $\mu \in \M(\sigma)$ has $\int \phi d\mu>-\infty$, then
\begin{align}\label{ref:dualvarprin}
h(\mu) + \int \phi \, d \mu = \inf \left\{P(\phi+g) - 	\int g \, d \mu : g \in \mathrm{UC}_{b,\rho}(\Sigma)	\right\}.
\end{align}
This dual variational principle asserts  that the subdifferentials of the pressure functional correspond to equilibrium measures.

We conclude this section by defining thermodynamic quantities at infinity. Let $(\Sigma, \sigma)$ be a transitive countable Markov shift and $\phi: \Sigma \to \R$ a potential of summable variations, the \emph{pressure at infinity} of $\phi$ is defined by
\begin{equation} \label{pre_inf}
P_{\infty}(\phi)=\sup_{(\mu_n)_n \mapsto 0} \limsup_{n \to \infty} \left(h_{\mu_n}(\sigma)  + \int \phi \, d\mu_n 	\right),
\end{equation}
where the supremum is taken over all sequences $(\mu_n)_n $ in $\M(\sigma)$ converging in cylinders to  the zero measure and  for which $\int \phi \, d\mu_n <\infty$ for every $n \in \N$. This notion has been studied with different degrees of generality by  R\"uhr and Sarig \cite{sr} and by Velozo \cite{v} (see also \cite{itv} where the notion of entropy at infinity was studied).  We say that  $\phi$ is \emph{strongly positive recurrent} (SPR) if  $P_\infty(\phi)<P(\phi)$. Potentials in this class are well behaved, for example, under mild integrability conditions  they have unique equilibrium state with strong ergodic properties (\cite[Theorem 1.4]{v} and \cite[Theorem 8.2]{sr}).


\section{Proofs of Theorem \ref{thm_ldp} and Theorem \ref{thm_equi}}

Let $(\Sigma,\sigma)$ be a topologically mixing countable Markov shift. In Section~\ref{sec:test} we introduced its compactification $(\bar\Sigma,\bar\sigma)$. Recall that there is a bijection between $\text{UC}_{b,\rho}(\Sigma)$ and $\text{C}(\bar\Sigma)$ via continuous extension and restriction. More precisely, every  $g\in\text{UC}_{b,\rho}(\Sigma)$ extends uniquely to a continuous function $\bar{g}:\bar\Sigma\to\R$. In particular, $\int \bar{g} \, d\mu$ is well defined for every $\mu\in \M(\bar\sigma)$, and if $\mu\in \M(\sigma)$ then $\int \bar{g} \, d\mu=\int g \, d\mu$. Moreover, the map $\eta\mapsto \int \bar{g} \, d\eta$ is continuous with respect to the weak$^\ast$ topology on $\M(\bar\sigma)$. 

 The proof strategy is similar to that used by  Kifer~\cite{kif} and Pollicott~\cite{po} once the phase space has been compactified. Let $\phi:\Sigma\to\R$ be a potential with summable variations and finite pressure.

\begin{definition}
The \emph{rate function} $I_\phi:\M(\bar\sigma) \to [0,\infty)$ is defined by
\begin{equation*}
I_\phi(\mu)= \sup \left\{ \int \bar{g}\, d \mu -P(g+\phi)  +P(\phi) : g \in \text{UC}_{b,\rho}(\Sigma)	\right\}.
\end{equation*}
\end{definition}

We note that $I_\phi$ is lower semicontinuous, since it can be expressed as the supremum of continuous functions. 
Setting $g=0$ immediately yields $I_\phi(\mu) \ge 0$ for every $\mu \in \M(\bar\sigma)$.

\subsection{Proof of Theorem \ref{thm_ldp}}

Set $\beta=\inf\{I_\phi(\eta):\eta\in \mathcal{K}\}$ and fix $\epsilon >0$. Note that
\begin{equation*}
\K \subset \left\{	 \eta \in \M(\bar\sigma): I_\phi(\eta) > \beta - \epsilon	 \right\}.
\end{equation*}
For each $g\in \text{UC}_{b,\rho}(\Sigma)$, define 
$$\mathcal{V}_g= \left\{	 \eta \in \M(\bar\sigma): \int \bar{g} \, d \eta - P(\phi+g) + P(\phi) > \beta - \epsilon	\right\}.$$  Since for every  $g\in \text{UC}_{b,\rho}(\Sigma)$ the map $\eta \mapsto  \int \bar{g}\, d \eta$ is continuous in $\M(\bar\sigma)$,  the set $\mathcal{V}_g$ is open in the weak$^\ast$ topology of $\M(\bar{\sigma})$. By the definition of $I_\phi(\eta)$,  we have that
$$\left\{	 \eta \in \M(\bar\sigma): I_\phi(\eta) > \beta - \epsilon	 \right\} =  \bigcup_{g \in \text{UC}_{b,\rho}(\Sigma)}\mathcal{V}_g.$$
Since $\K$ is compact and $\mathcal{K}\subset  \bigcup_{g \in \text{UC}_{b,\rho}(\Sigma)}\mathcal{V}_g$, there exists a finite sub-cover $\{\mathcal{V}_{g_1}, \dots , \mathcal{V}_{g_N}\}$. 

For a continuous function $\psi:\Sigma\to\R$, we set $S_n \psi(x)= \sum_{i=0}^{n-1} \psi(\sigma^i x)$ and  
$$Z_n(\psi)=\sum_{x\in \text{Per}_a(n)}e^{S_n\psi(x)}. $$
Let $g\in \text{UC}_{b,\rho}(\Sigma)$. Note that 
 \begin{align*}
\sum_{x\in \text{Per}_a(n), \mu_x\in \mathcal{V}_g}e^{S_n\phi(x)} = & \sum_{x\in \text{Per}_a(n), \mu_x\in \mathcal{V}_g}e^{S_n(\phi+g)(x)}e^{-S_ng(x)}\\
\le & \sum_{x\in \text{Per}_a(n), \mu_x\in \mathcal{V}_g}e^{S_n(\phi+g)(x)} e^{n(P(\phi)-P(\phi+g)-\beta+\epsilon)}\\
\le & Z_n(\phi+g)e^{-nP(\phi+g)}e^{n(P(\phi)-\beta+\epsilon)},
\end{align*}
where we used the definition of $\mathcal{V}_g$. It follows that,
 \begin{align*}
\frac{1}{Z_n(\phi)}\sum_{x\in \text{Per}_a(n), \mu_x\in \mathcal{V}_g}e^{S_n\phi(x)} \le & Z_n(\phi+g)e^{-nP(\phi+g)}Z_n(\phi)^{-1}e^{nP(\phi)}e^{n(-\beta+\epsilon)}.
\end{align*}
Set $\widehat{\mathcal{V}}=\bigcup_{i=1}^N \mathcal{V}_{g_i}$ and observe that,
 \begin{align*}
\frac{1}{Z_n(\phi)}\sum_{x\in \text{Per}_a(n), \mu_x\in \widehat{\mathcal{V}}}e^{S_n\phi(x)} \le & N Z_n(\phi+g)e^{-nP(\phi+g)}Z_n(\phi)^{-1}e^{nP(\phi)}e^{n(-\beta+\epsilon)}.
\end{align*}
From Lemma \ref{lem:c0}  we have that, 
 \begin{align*}
\limsup_{n\to\infty}\frac{1}{n}\log\bigg(\frac{1}{Z_n(\phi)}\sum_{x\in \text{Per}_a(n), \mu_x\in \widehat{\mathcal{V}}}e^{S_n\phi(x)} \bigg)\le -\beta+\epsilon.
\end{align*}
Since $\epsilon>0$ was arbitrary and $\mathcal{K}\subseteq \widehat{\mathcal{V}}$, the result follows.

\subsection{Proof of Theorem \ref{thm_equi}}  

Throughout this section assume that $(\Sigma,\sigma)$ is topologically mixing and that $\phi:\Sigma\to\R$ has summable variations, with $\sup\phi<\infty$, $P(\phi)<\infty$ and $s_\infty(\phi)<1$. 

\begin{lemma} \label{lem:l2}
If $\mu \in \M(\sigma)$ and $\int \phi \, d\mu>-\infty$, then $I_\phi( \mu)= P(\phi)-\big(h_\mu(\sigma)+\int \phi \, d\mu\big).$ 
\end{lemma}

\begin{proof}
    This follows directly from the dual variational principle for the pressure on countable Markov shifts proved in \cite{iv2} (see equation (\ref{ref:dualvarprin})). 
\end{proof}

\subsubsection{Proof of Theorem \ref{thm_equi}(a)}
From now on we further assume that $P_\infty(\phi)<P(\phi)$, that is, $\phi$ is SPR. In this case  $\phi$ admits a unique equilibrium state of $\phi$, denoted by $\mu_\phi$  

\begin{lemma}\label{lem:=0} If $\mu\in\M(\bar\sigma)$ satisfies that $I_\phi(\mu)=0$, then $\mu=\mu_\phi$.
\end{lemma}

\begin{proof} By definition,   for every $g\in  \text{UC}_{b,\rho}(\Sigma)$ we have that  $\int \bar{g} \,d\mu-P(g+\phi)+P(\phi)\le 0$.  Equivalently, 
$$\int \bar{g} \, d\mu\le P(\phi+g)-P(\phi),$$
for every $g\in  \text{UC}_{b,\rho}(\Sigma)$. In particular, we have that 
$$\int \bar{g} \, d\mu\le \frac{P(\phi+tg)-P(\phi)}{t},$$
for every $t>0$. It follows by the formula of the derivative of the pressure (see \cite[Theorem 3.2]{sr} and \cite{ijv}), that $\int \bar{g} \, d\mu\le \int g \, d \mu_\phi$, for every $g\in\text{UC}_{b,\rho}(\Sigma)$. Replacing $g$ with $-g$ we obtain that  $\int \bar{g} \, d\mu= \int g \, d \mu_\phi$, for every $g\in\text{UC}_{b,\rho}(\Sigma)$, and therefore $\mu=\mu_\phi$. 
\end{proof}

\begin{corollary} \label{cor:obs} We have
\begin{enumerate}
\item If $\mathcal{K}$ is a closed subset of $\M(\bar\sigma)$ such that $\mu_\phi\notin \mathcal{K}$, then 
$\inf\{I_\phi(\mu):\mu\in \mathcal{K}\}>0.$ 
\item If $\mathcal{K}$ is a closed subset of $\M(\sigma)$ such that $\mu_\phi\notin \mathcal{K}$, then 
$\inf\{I_\phi(\mu):\mu\in \mathcal{K}\}>0.$
\end{enumerate}
\end{corollary}
\begin{proof}
For the first part, suppose that $\inf\{I_\phi(\mu):\mu\in \mathcal{K}\}=0$. Since $I_\phi$ is lower semicontinuous and $\mathcal{K}$ is compact, there exists $\mu' \in \mathcal{K}$ such that $I_\phi(\mu')=0$. By Lemma~\ref{lem:=0}, this implies $\mu'=\mu_{\phi}$, so $\mu_\phi\in \mathcal{K}$. This contradicts the assumption. 
It is a consequence of Remark \ref{rem:weak=cyl} that $\mu_\phi\notin\overline{\mathcal{K}}$, where $\overline{\mathcal{K}}$ is the closure of ${\mathcal{K}}$ in $\M(\bar\sigma)$. To prove the second part  now  use part (a).
\end{proof}

\begin{proof}[Proof of Theorem \ref{thm_equi}(a)]
Let $g\in \text{UC}_{b,\rho}(\Sigma)$ and $\epsilon>0$. Define 
$$U=U(g,\epsilon)=\bigg\{\eta  \in\M(\bar\sigma):\bigg|\int \bar{g} \, d\eta-\int g \, d\mu_\phi\bigg|< \epsilon\bigg\}.$$ 
Since finite intersections of such sets form a basis for the weak$^\ast$ topology, it suffices to show that $\mu_n \in U(g,2\varepsilon)$ for all sufficiently large $n$. Consider the complement  of $U$: 
$$\K=\left\{\eta \in \M(\bar\sigma):\bigg|\int \bar{g} \, d\eta-\int g \, d\mu_\phi \bigg| \ge \epsilon\right\},$$
which is closed in $\M(\bar\sigma)$.  Since $\mu_\phi \notin \K$, Corollary \ref{cor:obs} yields
 $$\inf \{I_\phi(\mu):\mu\in \K\}=\beta>0.$$  
Choose $\beta_0\in(0,\beta)$. By Theorem \ref{thm_ldp} there exists $N_1 \in \N$ such that if $n\ge N_1$ then 
$$\sum_{x\in \text{Per}_a(n), \mu_x\in \K}e^{S_n\phi(x)}\le e^{-\beta_0 n}\sum_{x\in \text{Per}_a(n)}e^{S_n\phi(x)}.$$
By definition of the Gurevich pressure, our assumption is equivalent to
$$\lim_{n\to\infty}\frac{1}{n}\log\bigg(\frac{\sum_{x\in A_n}e^{S_n\phi(x)}}{\sum_{x\in \text{Per}_a(n)}e^{S_n\phi(x)}}\bigg)=0.$$
There exists $N_2 \in \N$ such that if $n\ge N_2$, then
$$ \sum_{x\in \text{Per}_a(n)}e^{S_n\phi(x)}\le e^{\beta_0 n/2} \sum_{x\in A_n}e^{S_n\phi(x)}.$$
Assume that $n\ge\max\{N_1,N_2\}$. Combining these inequalities we obtain
\begin{align}\label{eq:decay} 
\sum_{x\in A_n,\mu_x\in \K}e^{S_n\phi(x)}\le e^{-\beta_0n/2}\sum_{x\in A_n}e^{S_n\phi(x)}.
\end{align}
Recall that by definition 
$$\mu_n=\frac{1}{\sum_{x\in A_n}e^{S_n\phi(x)}}\sum_{x\in A_n}e^{S_n\phi(x)}\mu_x,$$
hence
$$\int g \, d\mu_n=\frac{1}{\sum_{x\in A_n}e^{S_n\phi(x)}}\sum_{x\in A_n}e^{S_n\phi(x)}\int g \, d\mu_x.$$
It follows from (\ref{eq:decay}) that
\begin{align}\label{eq:small}
\frac{1}{\sum_{x\in A_n}e^{S_n\phi(x)}}\sum_{x\in A_n,\mu_x\in \K}e^{S_n\phi(x)}\int g \, d\mu_x & \le \|g\|e^{-\beta_0 n/2}.
\end{align}
Additionally, by definition of $U$, we obtain
\begin{align}\label{eq:big}
\bigg|\sum_{x\in A_n,\mu_x\in U}e^{S_n\phi(x)}\int g \, d\mu_x-\sum_{x\in A_n,\mu_x\in U}e^{S_n\phi(x)}\int g \, d\mu_\phi\bigg| \le \epsilon \sum_{x\in A_n,\mu_x\in U}e^{S_n\phi(x)}.
\end{align}
Therefore, by (\ref{eq:small}) and (\ref{eq:big})
$$\bigg|\int g \, d\mu_n-\int g \, d\mu_\phi\bigg|\le \epsilon+\|g\|e^{-\beta_0n/2}.$$
We conclude that $\bigg|\int g \, d\mu_n-\int g \, d\mu_\phi\bigg|\le 2\epsilon$ for sufficiently large $n$. The conclusion now follows from Remark \ref{rem:weak=cyl}.
\end{proof}

\subsubsection{Proof of Theorem \ref{thm_equi}(b)} From now on we further assume that $\phi$ does not have any equilibrium state. In other words, $\phi$ is transient or null recurrent (see \cite{sa1} for precise definitions). 

Let $a\in\N$. The induced system over the cylinder $[a]$ is a full shift in a countable alphabet $\Sigma_F$ with  first return time $r_a:\Sigma_F\to\N$ (details of this construction can be found in \cite{sa2}). Denote by $\tilde{\phi}$ the induced potential of $\phi$ and $P_F(\cdot)$ denote the pressure on $\Sigma_F$. Set $$p_a(\phi)=\sup\{t:P_F(\tilde{\phi}+t r_a)<\infty\}.$$ The $a$-discriminant of $\phi$ is defined as $$\Delta_a(\phi)=P_F(\tilde{\phi}+p_a(\phi)r_a).$$ It is proved in the discriminant theorem \cite[Theorem 2]{sa2} that $\phi$ is transient if and only if $\Delta_a(\phi)<0$ and that in this case $P(\phi)=p_a(\phi)$. Note that $$p_a(\phi+s\chi_{[a]})=\sup\{t:P_F(\tilde{\phi}+tr_a)+s<\infty\}=p_a(\phi)$$ and that $\Delta_a(\phi+s\chi_{[a]})=\Delta_a(\phi)+s$.  In particular $\phi+s\chi_{[a]}$ is transient for $s$ small enough, and therefore 
\begin{align}\label{eq:trans}
P(\phi+s\chi_{[a]})=-p_a(\phi+s\chi_{[a]})=-p_a(\phi)=P(\phi).
\end{align}

It also follows by the discriminant theorem that $\Delta_a(\phi)\ge 0$ if and only if $\phi$ is recurrent, and in this case $P(\phi)$ satisfies that $P_F(\tilde{\phi}-P(\phi) r_a)=0$. Hence,  $\phi+t\chi_{[a]}$ is recurrent for $t\ge 0$ and therefore  
\begin{align}\label{for:111}
    0=P_F(\widetilde{\phi+t\chi_{[a]}})-P(\phi+t\chi_{[a]})r_a)=P_F(\tilde{\phi}-P(\phi+t\chi_{[a]})r_a)+t.
\end{align}

\begin{lemma}\label{lem:=01}
If $\mu\in\M(\bar\sigma)$ satisfies that $I_\phi(\mu)=0$, then $\mu(\Sigma)=0$. Equivalently, $\mu=\delta_{\bar{\infty}}$.
\end{lemma}

\begin{proof} By definition,   for every $g\in  \text{UC}_{b,\rho}(\Sigma)$ we have that  $\int \bar{g} \,d\mu-P(g+\phi)+P(\phi)\le 0$.  Equivalently, 
$$\int \bar{g} \, d\mu\le P(\phi+g)-P(\phi),$$
for every $g\in  \text{UC}_{b,\rho}(\Sigma)$. Fix $a\in\N$. Define $q(t)=P(\phi+t\chi_{[a]})-P(\phi)$. Note that
\begin{align}\label{ineq:111}
    \mu([a])\le \frac{q(t)}{t},
\end{align}
for every $t>0$. We will prove that $\mu([a])=0$. We have two cases to analyze: 

Case 1 ($\phi$ is transient). It follows by  (\ref{eq:trans}) and (\ref{ineq:111}) that for $t>0$ small we have that $\mu([a])=0$.

Case 2 ($\phi$ is null recurrent). By  (\ref{for:111}) we have that 
\begin{align*}
    \frac{t}{q(t)}&=\frac{P_F(\tilde{\phi}-P(\phi)r_a)-P_F(\tilde{\phi}-P(\phi+t\chi_{[a]})r_a)}{q(t)}\\
    &=\frac{P_F(\tilde{\phi}-P(\phi)r_a)-P_F(\tilde{\phi}-P(\phi)r_a-q(t)r_a)}{q(t)}.
\end{align*}
Thus, 
\[
\lim_{t\to 0^+}\frac{t}{q(t)}=-\frac{d}{ds}P_F(\tilde{\phi}-sr_a)\bigg|_{s=P(\phi)^+}.
\]
It is known that if $\phi$ is null recurrent then $\frac{d}{ds}P_F(\tilde{\phi}-sr_a)\big|_{s=P(\phi)^+}=-\infty$. We conclude that $\lim_{t\to 0^+}\frac{q(t)}{t}=0$. It follows by inequality (\ref{ineq:111}) that $\mu([a])=0$ for every $a\in\N$.
\end{proof}

The exact same proof of Corollary \ref{cor:obs} implies that 
\begin{corollary} 
If $\mathcal{K}$ is a closed subset of $\M(\bar\sigma)$ such that $\delta_{\bar\infty}\notin\mathcal{K}$, then  $$\inf\{I_\phi(\mu):\mu\in \mathcal{K}\}>0.$$
\end{corollary}

Finally, note that with these ingredients the proof of Theorem \ref{thm_equi}(b) is exactly the same as that of Theorem \ref{thm_equi}(a). In this case we consider $g\in \mathrm{UC}_{b,\rho}^0(\Sigma)$, $\epsilon>0$ and define
$$U=\left\{\eta\in\M(\bar\sigma):\left|\int \bar g\,d\eta\right|<\epsilon\right\}.$$
Then, let
$$\K=\left\{\eta\in\M(\bar\sigma):\left|\int \bar g\,d\eta\right|\ge\epsilon\right\},$$
which is closed in $\M(\bar\sigma)$. Arguing as in part (a), with $\delta_{\bar\infty}$ in place of $\mu_\phi$, we obtain
$$\left|\int g\,d\mu_n\right|\le \epsilon+\|g\|_\infty e^{-\beta_0n/2}$$
for all sufficiently large $n$. Hence $\mu_n\to\delta_{\bar\infty}$ in the weak$^\ast$ topology by Remark \ref{rem:weak=cyl}. This is equivalent to say that $(\mu_n)_n$ converges on cylinders to the zero measure (see  \cite[Theorem 3.5]{iv2}).

\section{Suspension semi-flows}\label{sec:sus}

In this section, we provide the necessary background—together with new results—on the corresponding thermodynamic formalism and the cylinder topology for suspension semi-flows over countable Markov shifts.

Let $(\Sigma, \sigma)$ be  a countable Markov shift and $\tau: \Sigma \to (0,\infty)$  a continuous function bounded away from zero, that is, $\inf \tau>0$. Consider the space
\begin{equation*}\label{eq:flow phase }
Y= \{ (x,t)\in \Sigma  \times \R \colon 0 \le t \le\tau(x)\}/\sim
\end{equation*}
where $(x,\tau(x))\sim (\sigma(x),0)$ for
each $x\in \Sigma $. The suspension semi-flow over $(\Sigma,\sigma)$
with roof function $\tau$ is the semi-flow $\Theta = (\theta_t)_{t \ge 0}$ on $Y$ defined by
$\theta_t(x,s)= (x,s+t)$ whenever $s+t\in[0,\tau(x)]$. In particular, $ \theta_{\tau(x)}(x,0)= (\sigma(x),0)$. 

We consider the following class of roof functions
$$\Psi=\{\tau:\Sigma\to\R \,| \,\tau \text{ has summable variations}, \inf\tau>0, \text{ and }P(-\tau)<\infty\}.$$

Throughout this work, we assume that $(Y,\Theta)$ has finite entropy (for precise definitions see Section \ref{sec:thm_susp}).  In this setting,  the suspension flow $(Y,\Theta)$ has finite entropy if and only if there exists $t_0>0$ such that $P(-t_0\tau)<\infty$. Accordingly, it is convenient, and sufficient, to restrict attention to the class $\Psi$ of roof functions.

\subsection{The space of invariant sub-probability measures} \label{sm}

Denote by $\M_{\leq 1}(\Theta)$ the space of Borel $\Theta$-invariant sub-probability measures on $Y$, that is, Borel measures $\nu$ such that 
\[
\nu(Y)\in[0,1]
\quad \text{and} \quad 
\nu(\theta_t^{-1}(A))=\nu(A) \quad \text{for all } t\geq 0 \text{ and all Borel sets } A\subset Y.
\]
We write $\M(\Theta)$ for the space of Borel $\Theta$-invariant probability measures on $Y$. Let $\text{C}_b(Y)$ denote the space of continuous bounded functions on $Y$.  On $\M_{\leq 1}(\Theta)$, we will consider the notions of convergence induced by the weak$^\ast$ and by the cylinder topologies.

A sequence of measures $(\nu_n)_{n} \subset \M_{\leq 1}(\Theta)$ is said to converge to 
$\nu \in \M_{\leq 1}(\Theta)$ in the \emph{weak$^\ast$ topology} if 
\[
\lim_{n \to \infty} \int f \, d\nu_n \;=\; \int f \, d\nu 
\quad \text{for every } f \in \text{C}_b(Y).
\]
Observe that if $(\nu_n)$ converges to $\nu$ in the weak$^\ast$ topology, then 
\[
\lim_{n \to \infty} \nu_n(Y) \;=\; \nu(Y).
\]
In particular, probability measures converge to probability measures. 
In general, however, the spaces $\M(\Theta)$ and $\M_{\leq 1}(\Theta)$ endowed with the weak$^\ast$ topology are non-compact.

Let $c \in \R^+$ be such that $\inf \tau > c$. 
For $(\nu_n)_{n}$ and $\nu$ measures in $\M_{\leq 1}(\Theta)$, 
we say that $(\nu_n)_{n}$ \emph{converges on cylinders} to $\nu$ if, for every cylinder $C \subset \Sigma$, 
\[
\lim_{n \to \infty} \nu_n(C \times [0,c]) \;=\; \nu(C \times [0,c]).
\]

The \emph{cylinder topology} is the topology that induces this notion of convergence. 
This topology is metrizable (see \cite[Lemma~6.6]{iv}). 
Moreover, if $(\nu_n)_{n}, \nu \in \M(\Theta)$ are probability measures, 
then $(\nu_n)_{n}$ converges on cylinders to $\nu$ if and only if it converges to $\nu$ in the weak$^\ast$ topology 
(see \cite[Lemma~6.7]{iv}). 
In other words, the cylinder topology restricts to the weak$^\ast$ topology on $\M(\Theta)$. 
Unlike the weak$^\ast$ topology, however, in the cylinder topology a sequence of probability measures may converge to 
a sub-probability measure; equivalently, the cylinder topology allows the escape of mass.

A fundamental property of the cylinder topology is that it makes the space of invariant sub-probability measures a compact space.

\begin{remark} \label{thm_compacidad}

Since we are assuming that $(Y,\Theta)$ has finite entropy, there exists $t_0>0$ such that $P(-t_0\tau)<\infty$. It follows from \cite[Theorem 1.6]{v} that the space of invariant sub-probability measures of the suspension semi-flow over $(\Sigma,\sigma)$ with roof function $t_0\tau$ is compact with respect to the cylinder topology. Since the spaces of invariant sub-probability measures of the suspension flows with roof functions $\tau$ and $t_0\tau$ are canonically identified, and homeomorphic, we conclude that the space $\mathcal{M}_{\leq 1}(\Theta)$ is compact metric space with respect to the cylinder topology. Furthermore,  $\mathcal{M}(\Theta)$ is a dense subset of $\mathcal{M}_{\leq 1}(\Theta)$.
\end{remark}

\subsection{Relationship to the base dynamics and the space of test functions} \label{sec:test_flow}
There is a close relationship  between the space $\M_{\leq 1}(\Theta)$ and the space of shift invariant probability measures  $\M(\sigma)$. Indeed, 
let
\begin{equation*}
\M_\tau(\sigma)= \left\{ \mu \in \mathcal{M}(\sigma): \int \tau \, d \mu < \infty \right\}.
\end{equation*}
Denote by $\L$ the  one-dimensional Lebesgue measure. A classical result by Ambrose and Kakutani \cite{ak} states that if $\mu \in \M_\tau(\sigma)$, then 
\begin{equation*} \label{ak} 
AK(\mu):=\frac{(\mu \times \L)|_{Y} }{(\mu \times \L)(Y)} \in \M(\Theta).
\end{equation*}
Note that by Fubini's theorem $(\mu \times \L)(Y)=\int \tau \, d\mu$. If $\tau$ is bounded away from zero, then the map $AK$ defines a bijection between $\M_\tau(\sigma)$ and $\M(\Theta)$.  

Every measure in $\M_{\leq 1}(\Theta)$ can be written as $\lambda \nu$ with $\nu \in \M(\Theta)$ and $\lambda \in [0,1]$. In particular, each measure in $\M_{\leq 1}(\Theta)$ can be expressed as $\lambda \, AK(\mu)$, with $\mu \in \M_\tau(\sigma)$ and $\lambda \in [0,1]$.

We can relate the entropy and integrals of measures in $\M(\Theta)$ to those of the corresponding measures in $\M_\tau(\sigma)$. 
By Abramov's formula \cite{a}, 
\[
h_\nu(\Theta) = \frac{h_\mu(\sigma)}{\int \tau \, d\mu},
\]  
where $\nu = AK(\mu)$ and $h_\nu(\Theta)$ denotes the measure-theoretic entropy of $\nu$ with respect to the time-one map $\theta_1$.  Moreover, by Kac's formula, if $f : Y \to \R$ is a measurable function, then
\[
\int_Y f \, d\nu = \frac{\int_\Sigma \Delta_f \, d\mu}{\int_\Sigma \tau \, d\mu},
\]  
where $\Delta_f : \Sigma \to \R$ is defined by
\[
\Delta_f(x) := \int_0^{\tau(x)} f(x,t) \, dt, \quad \text{for every } x \in \Sigma.
\]

\begin{remark}\label{rem:ex} Given a continuous function $\phi:\Sigma\to\R$, there exists a continuous function $f:Y\to\R$ such that $\Delta_f=\phi$. For instance, consider $$f(x,t)=\frac{\phi(x)}{\tau(x)}\psi'\bigg(\frac{t}{\tau(x)}\bigg),$$ for every $x\in \Sigma$ and $t\in [0,\tau(x)]$, where $\psi:[0,1]\to[0,1]$ is any nondecreasing $C^1$ function such that $\psi(0)=0$, $\psi(1)=1$, and $\psi'(0)=\psi'(1)=0$, see \cite{brw} for details. 
\end{remark}

The following result connects convergence on cylinders in $\M_{\leq 1}(\Theta)$ with properties of the corresponding sequences in $\M_\tau(\sigma)$.

\begin{lemma}[{\cite[Lemma 8.1.1]{v}}] \label{cyl_conver}
Let $(\nu_n)_n,\nu$ be measures in  $\M(\Theta)$, and set $\mu_n=(AK)^{-1}(\nu_n), \mu=(AK)^{-1}(\nu)$. Then the following statements are equivalent:
\begin{enumerate}
\item  $(\nu_n)_n$ converges on cylinders to the zero measure. 
\item Every subsequence of $(\mu_n)_n$ has a subsubsequence $(\mu_{n_k})_k$ which converges on cylinders to the zero measure or 
such that $\lim_{k\to\infty}\int \tau \, d\mu_{n_k}=\infty$.
\end{enumerate}
Similarly, the following statements are equivalent:
\begin{enumerate}
\item  $(\nu_n)_n$ converges on cylinders to $\lambda\nu$, where $\lambda\in (0,1]$. 
\item Every subsequence of $(\mu_n)_n$ has a subsubsequence $(\mu_{n_k})_k$ which converges on cylinders to $\lambda_1\mu$ and 
 $\lim_{k\to\infty}\int \tau \, d\mu_{n_k}=\lambda_2\int \tau \, d\mu$, for some $\lambda_1\in (0,1]$ and $\lambda_2\in [1,\infty)$ satisfying   $\lambda=\lambda_1/\lambda_2$.
\end{enumerate}
\end{lemma}

\begin{remark}\label{rem:convprob} Let $(\nu_n)_n$ and $\nu$ be measures in  $\M(\Theta)$. Set $\mu_n=(AK)^{-1}(\nu_n)$ and $\mu=(AK)^{-1}(\nu)$. Note that $(\nu_n)_n$ converges on cylinders to $\nu$ if and only if $(\mu_n)_n$ converges on cylinders to $\mu$, and $\lim_{n\to\infty}\int \tau d\mu_n=\int \tau d\mu$. Since the cylinder topology coincides with the weak$^\ast$ topology in the absence of escape of mass, this statement can equivalently be formulated in terms of weak$^\ast$ convergence.
\end{remark}

We conclude this section by defining a class of test functions for the cylinder topology on suspension semi-flows whose roof functions belong to the class $\Psi$. Let
\begin{equation}\label{eq:c0Y}
\text{C}_0(Y)= \left\{ g:Y \to\R \text{ continuous and bounded}: \Delta_g\in \text{UC}_{b,\rho}^0(\Sigma)	\right\},
\end{equation}
where $\text{UC}_{b,\rho}^0(\Sigma)$ is given in (\ref{eq:test1}). 

\begin{lemma}\label{prop:test}
Let $(Y,\Theta)$ be the suspension semi-flow over $(\Sigma, \sigma)$ with roof function $\tau\in \Psi$. If $(\nu_n)_n, \nu$  are measures in $\M(\Theta)$ and $\lambda \in [0,1]$ then, $(\nu_n)_n$ converges on cylinders to $\lambda \nu$ if and only if for every $g\in \emph{C}_0(Y)$ we have
\begin{equation}  \label{eq:lema_test} 
\lim_{n \to \infty} \int g\, d \nu_n= \int g \, d (\lambda\nu).
\end{equation} 
 \end{lemma}

\begin{proof}
Let $\mu_n= (AK)^{-1}(\nu_n)$ and $\mu= (AK)^{-1}(\nu)$. Note that by Kac's formula, the limit in equation \eqref{eq:lema_test} can be rewritten as 
\begin{equation}\label{eq:test}
\lim_{n \to \infty}  \frac{\int \Delta_g \, d \mu_n }{\int \tau\, d \mu_n}= \lambda  \frac{\int \Delta_g \, d \mu}{\int \tau\, d \mu}.
 \end{equation}
The equivalence now follows directly from Lemma~\ref{lem:testsigma}, Lemma~\ref{cyl_conver} and the definition of $\text{C}_0(Y)$.
\end{proof}

\subsection{Thermodynamic formalism} \label{sec:thm_susp}

The notion of pressure for suspension semi-flows over countable Markov shifts has been studied in different degrees of generality, both on the class of systems and functions, see  \cite{bi,jkl, ke,sav}. We propose a variational formula since it can be applied to a very general class of functions. 

Let $(Y, \Theta)$ be a suspension semi-flow over a countable Markov shift and $f:Y \to \R$ a continuous function. The \emph{pressure} of $f$ is defined by
\begin{equation*}
P_{\Theta}(f)= \sup \left\{h(\nu) +  \int f \, d \nu: \nu \in \M(\Theta)  \text{ and }	   \int f \, d \nu >  -\infty		\right\}.
\end{equation*}

Let $a\in\N$. For $T\ge T'\ge 0$ we define 
$$\text{FPer}_a(T',T)=\big\{(x;s)\in[a]\times [T',T]:\theta_s(x,0)=(x,0)\big\}.$$
The set $\text{FPer}_a(T',T)$ can be identified with the collection of periodic orbits starting in $[a] \times \{0\}\subset Y$ whose periods lie in the interval $[T',T]$. We also define $$\text{FPer}_a(T)=\big\{(x;s)\in[a]\times [0,T]:\theta_s(x,0)=(x,0)\big\},$$
where we consider all orbits of length $\le T$. Finally, define $$\text{FPer}(T',T)=\bigcup_{a\in\N}\text{FPer}_a(T',T),$$ where we do not restrict the first coordinate to be fixed. 

The following results describe some of the basic properties of the pressure,  see \cite[Theorem 1.2]{jkl} and also
\cite{bi, ke,sav}.

\begin{theorem} \label{thm: flow pres}
Let $(\Sigma, \sigma)$ be a topologically mixing countable Markov shift and $\tau:\Sigma \to \R$ a function of summable variations and bounded away from zero. Let $(Y, \Theta)$ be the associated suspension semi-flow. Let $f:Y \to \R$ be a continuous function such that $\Delta_f$  has summable variations. Let  $a\in\N$ and $d>0$. Then, 
\begin{eqnarray*}
P_{\Theta}(f)&=&\lim_{T \to \infty} \frac{1}{T} \log \left(\sum_{(x;s)\in \mathrm{FPer}_a(T-d,T)} \exp\left( \int_0^s f(\theta_t(x,0))    \, dt \right)\right)\\
&=&\inf\{t \in \R : P (\Delta_f- t \tau) \leq 0\} =\sup \{t \in \R : P (\Delta_f- t \tau) \geq 0\} \\
&=& \sup \{ P_{\Theta|K}(f) : K \text{ is a compact and } \Theta\text{-invariant set} \}.
 \end{eqnarray*}
\end{theorem}

\begin{remark}\label{rem:positivepressure}
In \cite{ke} the above result is stated with the summation running over  $\mathrm{FPer}_a(T)$. 
Since $\mathrm{FPer}_a(T)$ is increasing in $T$, this only makes sense when $P_\Theta(f)\ge0$. On the other hand, it follows from Theorem \ref{thm: flow pres} that if $P_\Theta(f)\ge 0$, then for every $\epsilon>0$ there exists $T_0$ such that if $T\ge T_0$, then
$$e^{T(P_\Theta(f)+\epsilon)}\ge \sum_{(x;s)\in \mathrm{FPer}_a(T-d,T)} \exp\big( \int_0^s f(\theta_t(x,0))    \, dt\big),$$
and therefore 
$$C+e^{T(P_\Theta(f)+\epsilon)}\sum_{n\ge 0}e^{-nd(P_\Theta(f)+\epsilon)}\ge \sum_{(x;s)\in \mathrm{FPer}_a(T)} \exp\big( \int_0^s f(\theta_t(x,0))    \, dt\big),$$
where $C=\sum_{(x;s)\in \mathrm{FPer}_a(T_0)} \exp( \int_0^s f(\theta_t(x,0))    \, dt)$. We conclude that
$$P_{\Theta}(f)=\lim_{T \to \infty} \frac{1}{T} \log \left(\sum_{(x;s)\in \mathrm{FPer}_a(T)} \exp\left( \int_0^s f(\theta_t(x,0))    \, dt \right)\right).$$

\end{remark}

\begin{remark}\label{rmk:bipPer}

If, in addition to the assumptions in Theorem \ref{thm: flow pres}, the system $(\Sigma, \sigma)$ has the BIP property, then the following equality hold:
\begin{eqnarray*}
P_{\Theta}(f)&=&\lim_{T \to \infty} \frac{1}{T} \log \left(\sum_{(x;s)\in\text{FPer}(T-d,T)} \exp\left( \int_0^s f(\theta_t(x,0)) \,  dt \right) \right). \end{eqnarray*}
In other words, we sum over all periodic orbits rather than only those intersecting $[a] \times\{0\}$ (see \cite[Corollary 3.23]{jkl}).
\end{remark}

 Under the assumptions of Theorem \ref{thm: flow pres}, the potential 
$f:Y \to \R$ admits at most one equilibrium state. Moreover, $f$ has a 
unique equilibrium state if and only if 
\[
P(\Delta_f - P_{\Theta}(f)\tau) = 0,
\]
the potential $\Delta_f - P_{\Theta}(f)\tau$ admits a unique equilibrium state 
$\mu \in \M_{\tau}(\sigma)$, and $\int \Delta_f \, d\mu < \infty.$
In this case, the equilibrium state of $f$ is given by $AK(\mu)$; 
see \cite[Theorem~4]{bi} and \cite[Theorems~3.4 and 3.5]{ijt}.

The \emph{entropy} of the suspension semi-flow, denoted by $h_{top}(\Theta)$,  is defined as the pressure of the constant function equal to zero. Note that by Theorem \ref{thm: flow pres}, 
\begin{eqnarray*}
h_{top}(\Theta) &=&\lim_{T \to \infty} \frac{1}{T} \log \#\text{FPer}_a(T-d,T) \\
&=& \inf \{ t \in \R : P(-t \tau) \leq 0 \}.
\end{eqnarray*}
In particular, if $\tau\in\Psi$, then $(Y,\Theta)$  has finite entropy.

The following result is analogous to Lemma \ref{lem:c0} and will be used in the proof of the large deviation upper bound for the suspension semi-flow.

\begin{lemma}\label{lem:presC0} Let $(\Sigma, \sigma)$ be a topologically mixing countable Markov shift and $\tau:\Sigma \to \R$ a function of summable variations and bounded away from zero. Let $(Y, \Theta)$ be the associated suspension semi-flow and  $f, g: Y \to \R$  functions such that $\Delta_f$  has summable variations, $P_{\Theta}(f)<\infty$, and $\Delta_g\in \emph{UC}_{b,\rho}(\Sigma)$.  Let  $a\in\N$ and $d>0$, then,
\begin{align}\label{for:11}
P_{\Theta}(f+g) = \lim_{T \to \infty} \frac{1}{T} \log \left(\sum_{(x;s)\in \emph{FPer}_a(T-d,T)} \exp\left( \int_0^s (f+g)(\theta_t(x,0)) \, dt \right) \right).
 \end{align}
\end{lemma}

\begin{proof}  By Lemma \ref{lem:appr}, we can approximate $\Delta_g$ by potentials with summable variations. Fix $\epsilon>0$, and let $h:\Sigma\to\mathbb{R}$ be a function with summable variations such that $\|h - \Delta_g\|_\infty < \epsilon$. By Remark \ref{rem:ex} there exists  $g_0 :Y\to\R$ such that $\Delta_{g_0} = h$.  In particular, $\|\Delta_{g_0}- \Delta_g\|_\infty < \epsilon$ and,  by Theorem \ref{thm: flow pres}, formula (\ref{for:11}) holds for $f+g_0$.

By Kac's formula, 
$$\left|\int g \, d\nu -\int g_0  \, d\nu \right| \le \frac{1}{\inf \tau}\left\|\Delta_g-\Delta_{g_0} \right\|_\infty,$$ 
for every $\nu\in\M(\Theta)$. Therefore, 
\begin{align}\label{eq:g01}
\left|P_\Theta(f+g)-P_\Theta(f+g_0) \right|\le \frac{1}{\inf \tau}\epsilon.
\end{align}

Fix $c \in (0, \inf \tau)$ and let $(x, 0) \in Y$ be a point such that $\theta_s(x, 0) = (x, 0)$, for some $T-d < s \le T$. Then,  there exists $n \in \mathbb{N}$ such that $\sigma^n(x) = x$ and $S_n \tau(x) = s$. Note that $nc \le s$, hence $n \le \frac{T}{c}$. Observe that 
$\int_0^s g(\theta_t(x,0)) \, dt=S_n\Delta_g(x)$, and $\int_0^s g_0(\theta_t(x,0)) \, dt=S_n \Delta_{g_0}(x).$ It follows that 
$$\left|\int_0^s (f+g)(\theta_t(x,0)) \, dt-\int_0^s (f+g_0)(\theta_t(x,0)) \, dt\right|\le n\epsilon\le \frac{\epsilon T}{c}.$$
Using this inequality and the fact that formula (\ref{for:11}) holds for $f+g_0$, we obtain
$$\bigg|\limsup_{T \to \infty} \frac{1}{T} \log \left(\sum_{(x;s)\in \text{FPer}_a(T-d,T)} \exp\left( \int_0^s (f+g)(\theta_t(x,0)) \, dt \right) \right)-P_\Theta(f+g_0)\bigg|\le \frac{\epsilon}{c}$$
Combining with (\ref{eq:g01}), we get
\begin{align*}
\left|\limsup_{T \to \infty} \frac{1}{T} \log \left(\sum_{(x;s)\in \text{FPer}_a(T-d,T)} \exp\left( \int_0^s (f+g)(\theta_t(x,0)) \, dt \right) \right)-P_\Theta(f+g)\right| &\le&\\ \epsilon\left(\frac{1}{c}+\frac{1}{\inf \tau}\right).
\end{align*}
An analogous bound holds for the $\liminf$. Since $\epsilon>0$ is arbitrary, we conclude that formula (\ref{for:11}) follows.
\end{proof}

\subsection{Entropy and pressure at infinity} \label{sec:pre_imfty_flow}
The notion of \emph{entropy at infinity}  for suspension semi-flows defined over countable Markov shifts was introduced by the authors in \cite{irv}. Measures the amount of disorder that measures that escape the system take with them. Recall that when a sequence of measures $(\nu_n)_n$ converges on cylinders to the zero measure we say that it converges on cylinders to $0$. The entropy at infinity of the suspension semi-flow $(Y,\Theta)$  is defined by
\begin{equation*}
h_{\infty}(\Theta)= \sup \left\{\limsup _{n \to \infty}h_{\nu_n}(\Theta): \nu_n \in \M(\Theta),  (\nu_n)_n \text{ converges on cylinders to  } 0 \right\}.
\end{equation*}
Note that, in general, $h_{top}(\Theta)\ge h_\infty(\Theta)$.  Analogously, we define the \emph{pressure at infinity} of a continuous function  $f:Y\to\R$  as follows
\begin{equation*}
P_{\Theta,\infty}(f)= \sup \left\{\limsup _{n \to \infty} \left( h_{\nu_n}(\Theta)+\int f \, d\nu_n \right) \right\},
\end{equation*}
where the supremum is taken over all sequences $(\nu_n)_n$  of $\Theta-$invariant probability measures  such that  $(\nu_n)_n$ converges on cylinders to the zero measure.

\begin{definition}  
Let  $f \in \text{C}_b(Y)$ be a potential such that $\Delta_f\in\mathrm{UC}_d(\Sigma)$ and $V_2(\Delta_f)<\infty$. We say that $f$ is \emph{strongly positive recurrent} (SPR) if $P_{\Theta,\infty}(f)<P_\Theta(f)$. If $f=0$ is SPR, we say that $(Y, \Theta)$ is strongly positive recurrent. 
\end{definition}

\begin{remark}\label{rem:forder}
It follows from Remark \ref{thm_compacidad} and \cite[Theorem 8.1]{v} that if $f$ is SPR, then $f$ admits at least one equilibrium state. Moreover, if $(\nu_n)_n$ is a sequence in $\mathcal{M}(\Theta)$ such that $$P_\Theta(f)=\lim_{n\to\infty}\bigg(h_{\nu_n}(\Theta)+\int f d\nu_n\bigg),$$ then $(\nu_n)_n$ has a subsequence that converges to an equilibrium state of $f$. If, in addition, $\Delta_f$ has summable variations, then the equilibrium state is unique (see  \cite[Theorems~3.4 and 3.5]{ijt}), and the sequence $(\nu_n)_n$ converges in the weak$^\ast$ topology to the unique equilibrium state.
\end{remark}

\begin{lemma}\label{lem:deri}  Let $(\Sigma,\sigma)$ be a transitive countable Markov shift, $\tau \in \Psi$ and $(Y,\Theta)$ be the associated suspension semi-flow.  Let $f\in \mathrm{C}_b(Y)$ be a SPR potential such that $\Delta_f$ has summable variations and $g\in \mathrm{C}_b(Y)$  such that $\Delta_g\in \emph{UC}_{b,\rho}(\Sigma)$. Define $h(t) =P _\Theta(f + tg)$. Then, $h'(0) = \int g \, d\nu_f$, where $\nu_f$ is the equilibrium state of $f$.
\end{lemma}

\begin{proof} By assumption $\Delta_g\in \text{UC}_{b,\rho}(\Sigma)$, and therefore $\Delta_g\in\mathrm{UC}_d(\Sigma)$ and $V_2(\Delta_g)<\infty$ (see Lemma \ref{lem:appr}). Note that there exists $\epsilon>0$ such that $f+tg$ is SPR for every $t\in (-\epsilon,\epsilon)$. Denote by $\nu_t$ an equilibrium state of $f+tg$ (see Remark \ref{rem:forder}). Note that 
\begin{align*}
\lim_{t\to 0} \left( h_{\nu_t}(\Theta)+\int f \, d\nu_t \right) &=\lim_{t\to 0} \left( h_{\nu_t}(\Theta)+\int (f+tg) \, d\nu_t \right)\\
&=\lim_{t\to 0} P_\Theta(f+tg)\\
&=P_\Theta(f)
\end{align*}
It follows by Remark \ref{rem:forder} that $(\nu_t)_t$ converges in the weak$^\ast$ topology to $\nu_0$ as $t\to 0$. In particular, we have that $\lim_{t\to0} \int g \, d\nu_t=\int g \, d\nu_0$. Note that by the variational principle have that if $t\in (0, \epsilon)$, then 
$$ \int g \, d\nu_0 \le \frac{P_\Theta(f+tg)-P_\Theta(f)}{t}\le \int g \, d\nu_t.$$
Similarly, if $t\in (-\epsilon,0)$, then 
$$ \int g \, d\nu_t \le \frac{P_\Theta(f+tg)-P_\Theta(f)}{t}\le \int g \, d\nu_0.$$
Taking limit $t\to 0$ we obtain the desired formula.
\end{proof}

\subsection{The dual variational principle and the class \texorpdfstring{$\cR$}{cR}} 
Motivated by its relevance in statistical mechanics, convex analysis has played a central role in the study of thermodynamic formalism. Through the dual variational principle, equilibrium sates arise as the subdifferentials of the pressure functional, a fact central to its applications. In this section, we prove the dual variational principle for suspension semi-flows whose roof function has uniform growth at infinity. 

 \begin{definition} \label{def:R} We say that $\tau:\Sigma\to (0,\infty)$ belongs to the class $\cR$ if the following properties hold: 
\begin{enumerate}
\item \label{Ta} $\tau$ is bounded away from zero and has summable variations,
\item \label{Tb} $\lim_{k\to \infty}\inf_{x\in [k]}\tau(x)=\infty.$
\end{enumerate}
\end{definition}

The class of suspension semi-flows with roof functions in $\cR$ was studied in \cite{iv} under a weaker regularity assumption. In this work, we consider roof functions with summable variations. This class arises naturally in the study of systems with the BIP property, as shown in the next lemma.

\begin{lemma} \label{lem:bip}
Let $(\Sigma,\sigma)$ be a topologically mixing countable Markov shift with the
BIP property. Let $\phi:\Sigma\to \R$ be a function with summable
variations satisfying $\inf \phi > 0$, $V_1(\phi)<\infty$ and $P(-\phi)<\infty$. Then
$\phi\in \mathcal{R}$.
\end{lemma}

\begin{proof}
Since $\phi$ is bounded away from zero and has summable variations by
assumption, it remains only to prove that
\[
        \lim_{k\to\infty}\inf_{x\in[k]}\phi(x)=+\infty .
\]
We argue by contradiction. Suppose that this limit does not hold. Then there
exist a constant $M>0$ and an infinite sequence of distinct symbols
$(a_j)_{j\ge1}$ such that for every $ j\ge1$ we have
\[  \inf_{x\in[a_j]}\phi(x)\le M.  \]
Since $V_1(\phi)<\infty$, it follows that for every $ j\ge1$
\[ \sup_{x\in[a_j]}\phi(x)\le M+V_1(\phi).       \]
We now use the BIP property and topological mixing. By BIP, there exists a
finite set $B$ of symbols such that for every symbol $a$ there are
$b(a),b'(a)\in B$ with
\[
        b(a)\,a\,b'(a)
\]
admissible. Since the shift is topologically mixing and $B$ is finite, there
exists an integer $p\ge1$ and a finite collection $\mathcal W$ of admissible
words of length $p$ such that, for every symbol $a$, one can choose
$w(a)\in\mathcal W$ with
\[
        a\,w(a)\,a
\]
admissible. Set
\[
        C_{\mathcal W}:=
        \max_{w\in\mathcal W}\sup_{x\in[w]} S_p\phi(x)<\infty .
\]
This quantity is finite because $\mathcal W$ is finite and $V_1(\phi)<\infty$. For each $j\ge1$, let $w_j:=w(a_j)$. The word $a_j\,w_j $
determines a periodic point $x_j$ of period $p+1$ whose itinerary repeats the
block $a_jw_j$. Along one period we have
\[
        S_{p+1}\phi(x_j)
        \le
        M+V_1(\phi)+C_{\mathcal W}.
\]
Hence, for every $n\ge1$,
\[
        S_{n(p+1)}\phi(x_j)
        \le
        n\bigl(M+V_1(\phi)+C_{\mathcal W}\bigr).
\]
Therefore
\[
        \exp\bigl(-S_{n(p+1)}\phi(x_j)\bigr)
        \ge
        \exp\left(-n\bigl(M+V_1(\phi)+C_{\mathcal W}\bigr)\right).
\]
Since the symbols $a_j$ are distinct, the periodic points $x_j$ are distinct.
Thus, for every $n\ge1$, the partition sum at time $n(p+1)$ satisfies
\[
\begin{aligned}
        Z_{n(p+1)}(-\phi)
        &\ge
        \sum_{j=1}^\infty
        \exp\bigl(-S_{n(p+1)}\phi(x_j)\bigr)       
        &\ge
        \sum_{j=1}^\infty
        \exp\left(-n\bigl(M+V_1(\phi)+C_{\mathcal W}\bigr)\right)
        =
        \infty .
\end{aligned}
\]
Consequently $P(-\phi)=\infty$, contradicting the hypothesis
$P(-\phi)<\infty$. This contradiction proves that
\[
        \lim_{k\to\infty}\inf_{x\in[k]}\phi(x)=+\infty .
\]
Therefore $\phi\in \mathcal{R}$.
\end{proof}

Let $(\Sigma,\sigma)$ be a topologically mixing countable Markov shift and $\tau\in \cR$. Denote by $(Y,\Theta)$ the associated suspension semi-flow. In \cite[Lemma 6.9 and Lemma 6.10]{iv}, it is shown that, in this setting,  the cylinder topology takes a simpler form, which allows one to consider a larger class of test functions. Define
\begin{equation*}
\text{C}^{\cR}_0(Y)= \left\{ g:Y \to\R \text{ continuous and bounded}: \Delta_g\in \text{UC}_{b,\rho}(\Sigma)	\right\}.
\end{equation*}

\begin{lemma}\label{prop:test3}
Suppose that $\tau\in \cR$. Let $(\nu_n)_n, \nu$ be measures in $\M(\Theta)$ and $\lambda \in [0,1]$.  Then, $(\nu_n)_n$ converges on cylinders to $\lambda \nu$ if and only if for every $g\in \emph{C}^{\cR}_0(Y)$ we have
\begin{equation*}
\lim_{n \to \infty} \int g\, d \nu_n= \int g \, d (\lambda\nu).
\end{equation*} 
 \end{lemma}
 
\begin{proof} Let $\nu_n= (\mu_n \times \text{Leb})/(\int \tau \, d \mu_n)$ and $\nu= (\mu \times \text{Leb})/(\int \tau \, d \mu)$.  By Kac's formula, the limit can be rewritten as \begin{equation}\label{eq:test2}
\lim_{n \to \infty}  \frac{\int \Delta_g \, d \mu_n }{\int \tau\, d \mu_n}= \lambda  \frac{\int \Delta_g \, d \mu}{\int \tau\, d \mu}.
 \end{equation}
Suppose $\lambda=0$. By \cite[Lemma 6.9]{iv}, $(\nu_n)_n$ converges on cylinders to the zero measure if and only if $\lim_{n\to\infty}\int \tau d\mu_n=\infty$. This implies that (\ref{eq:test}) holds for every $\Delta_g\in  \text{UC}_{b,\rho}(\Sigma)$. On the other hand, considering $\Delta_g=1$, we deduce the other implication. We now suppose that $\lambda\ne 0$. By \cite[Lemma 6.10]{iv}, $(\nu_n)_n$ converges on cylinders to $\lambda\nu$ if and only if $(\mu_n)_n$ converges in the weak$^\ast$ topology to $\mu$ and $\lim_{n\to\infty}\int \tau \, d\mu_n=\frac{1}{\lambda}\int \tau \, d\mu$. This implies that (\ref{eq:test}) holds for every $\Delta_g\in  \text{UC}_{b,\rho}(\Sigma)$. On the other hand, considering  $\Delta_g=1$, we deduce that $\lim_{n\to\infty}\int \tau \, d\mu_n=\frac{1}{\lambda}\int \tau \, d\mu$, and therefore $\lim_{n\to\infty}\int \Delta_g \, d\mu_n=\int \Delta_g \, d\mu$, for every $g\in\text{C}^{\cR}_0(Y)$. We conclude that $(\mu_n)_n$ converges in the weak$^\ast$ topology to $\mu$. \end{proof}

\begin{theorem} \label{thm:dual}
Let $(\Sigma, \sigma)$ be a transitive countable Markov shift  and  $\tau\in\Psi\cap \cR$. Let $(Y,\Theta)$ be the associated suspension semi-flow. If $f\in \mathrm{C}_b(Y)$ is such that $\Delta_f\in \mathrm{UC}_d(\Sigma)$ and $V_2(\Delta_f)<\infty$, then, for every $\nu \in \M(\Theta)$, we have
\begin{equation*}
h_\nu(\Theta)+\int f \, d\nu= \inf \left\{P_{\Theta}(f+g) - \int g \, d \nu :g\in C^\cR_0(Y)		\right\}.
\end{equation*}
\end{theorem}

\begin{proof}The inequality 
$$h_\nu(\Theta)+\int f \, d\nu \leq \inf \left\{P_{\Theta}(f+g) - \int g \, d \nu : g \in C^\cR_0(Y)			\right\},$$
follows directly from the definition. Fix $\epsilon>0$. To prove the reverse inequality, it suffices to construct a function $g \in C^\cR_0(Y)$ such that
\begin{equation} \label{eq:f}
h_\nu(\Theta) + \int (f+g) \, d \nu -  P_{\Theta}(f+g)\ge - \epsilon.
\end{equation}
Let  $\mu \in \M_\tau(\sigma)$ be such that $\nu=(\mu \times \L)/(\int \tau \, d \mu)$. Since $f$ is bounded above, there exists $t^*>0$ such that $P(\Delta_f-t^*\tau) <\infty$ and $s_\infty(\Delta_f-t^*\tau)<1$. It follows from the dual variational principle for countable Markov shifts (see  \cite[Theorem 1.4]{iv2}), applied to the potential $\Delta_f-t^*\tau$, that there exists $h \in \text{UC}_{b,\rho}(\Sigma)$ such that 
\begin{equation*} \label{eq:estrella} 
h_\mu(\sigma) + \int \left(\Delta_f-t^* \tau + h \right) \, d\mu  -P(\Delta_f-t^* \tau + h) \ge  - \epsilon \int \tau \, d\mu.
\end{equation*}
By Remark \ref{rem:ex}, there exists $g \in C^\cR_0(Y)$ such that $\Delta_{g}=  h - P(\Delta_f-t^* \tau + h)$.  Therefore,
$$h_\mu(\sigma) +\int(\Delta_f+\Delta_g-t^*\tau) d\mu \ge  -\epsilon \int \tau d\mu.$$
By Abramov's and Kac's formulae, we obtain that
$$h_\nu(\Theta) +\int (f+g) d\nu -t^* \ge  -\epsilon.$$
Note that $P_{\Theta}(f+g)=t^*$. Indeed, $P(\Delta_f+\Delta_{g}-t^*\tau)=0,$ and for every $s <0$ we have that
$P(\Delta_f+\Delta_{g} -(t^*+s) \tau) \geq  P(\Delta_f+\Delta_{g}-t^*\tau) + |s| \inf \tau >0$. We conclude that (\ref{eq:f}) holds.
\end{proof}

We conclude this section with a result that allows to compute the entropy at infinity for suspension semi-flows with base \emph{BIP}. Recall that $s_\infty(\phi)=\inf\{t:P(t\phi)<\infty\}$.

\begin{lemma}\label{lem:entropyinf}
Let $(\Sigma, \sigma)$ be a topologically mixing countable Markov shift with the \emph{BIP} property. Let $\tau:\Sigma\to(0,\infty)$ be a roof function with summable variations and bounded away from zero. Let $(Y,\Theta)$ be the associated suspension flow. Assume that $(Y,\Theta)$ has finite entropy. Then, $$h_\infty(\Theta)=s_\infty(-\tau).$$
\end{lemma}
\begin{proof}
Note that there exists $t_0>0$ such that $P(-t_0\tau)<\infty$. It follows from Lemma \ref{lem:bip} that $\tau\in\cR$. By \cite[Lemma 6.9]{iv} 
we have that a sequence $(\nu_n)_n$ in $\M(\Theta)$ converges on cylinders to the zero measure if and only if $\lim_{n\to\infty}\int \tau d (AK)^{-1}(\nu_n)=\infty$. The result now follows directly from \cite[Theorem 3.7]{irv}. 
\end{proof}

\section{Proofs of Theorem \ref{thm_ldp2} and Theorem \ref{thm_equi2}}

Let $(\Sigma,\sigma)$ be a topologically mixing countable Markov shift and  $\tau\in\Psi$. Denote by  $(Y,\Theta)$  the associated suspension semi-flow. Let $f\in \mathrm{C}_b(Y)$ be such that $\Delta_f$ has summable variations. 

\begin{definition} 
The rate function $I_f:\M_{\leq 1}(\Theta) \to [0,\infty)$ is defined by
\begin{equation*}
I_f(\nu)= \sup \left\{ \int g\, d \nu -P_\Theta(f+g)  +P_\Theta(f) : g \in \mathrm{C}_0(Y)	\right\}
\end{equation*}
\end{definition}

By Lemma \ref{prop:test}, the map $\nu \mapsto \int g \, d\nu$ is continuous on $\M_{\le 1}(\Theta)$ for every $g \in C_0(Y)$. This implies that $I_f$ is lower semicontinuous, since it is the supremum of continuous functions. Moreover, taking $g = 0$ yields $I_f(\nu) \ge 0$ for every $\nu \in \M_{\le 1}(\Theta)$.

For suspension semi-flows with roof functions in the class $\Psi \cap \cR$, the rate function can be described more explicitly. Indeed, it directly follows from Theorem \ref{thm:dual} that, 

\begin{corollary}\label{cor:R} If $\tau\in \Psi\cap\cR$, and $\nu\in \M(\Theta)$,  then $I_f(\nu)=P_\Theta(f)-h_\nu(\Theta)-\int f \, d\nu$.
\end{corollary}

\subsection{Proof of Theorem \ref{thm_ldp2}} 

We adapt the strategy used in the proof of Theorem  \ref{thm_ldp} with appropriate modifications. Set $\beta=\inf\{I_f(\nu):\nu\in \mathcal{K}\}$ and fix $\epsilon >0$. Note that
\begin{equation*}
\K \subset \left\{	 \nu \in \M_{\le 1}(\Theta): I_f(\nu) > \beta - \epsilon	 \right\}.
\end{equation*}
For each $g\in \mathrm{C}_0(Y)$, define 
$$\mathcal{V}_g= \left\{	 \nu \in \M_{\le 1}(\Theta): \int g \, d \nu - P_\Theta(f+g) + P_\Theta(f) > \beta - \epsilon	\right\}.$$  
Since for every  $g\in \mathrm{C}_0(Y)$ the map $\nu \mapsto  \int g \, d \nu$ is continuous in $\M_{\le 1}(\Theta)$,  the set $\mathcal{V}_g$ is open in the cylinder topology of $\M_{\le1}(\Theta)$. By the definition of $I_f(\nu)$,  we have that
$$\left\{	 \nu \in \M_{\le 1}(\Theta): I_f(\nu) > \beta - \epsilon	 \right\} =  \bigcup_{g \in \mathrm{C}_0(Y)}\mathcal{V}_g.$$
 Since $\K$ is compact and $\mathcal{K}\subset  \bigcup_{g \in \mathrm{C}_0(Y)}\mathcal{V}_g$, there exists a finite sub-cover $\{\mathcal{V}_{g_1}, \dots , \mathcal{V}_{g_N}\}$. 
For a continuous function $h:Y\to\R$ let,
$$Z_T(h)=\sum_{(x;s)\in \mathrm{FPer}_a(T-d,T)} \exp\left( \int_0^s h(\theta_t(x,0))  \, dt \right). $$
Set $r=\max_{1\le i\le N}\{P_\Theta(f)-P_\Theta(f+g_i)-\beta+\epsilon\}$ and let $g\in\{g_1,\ldots,g_N\}$. Define 
$$A_T(g)=\{(x;s)\in\mathrm{FPer}_a(T-d,T):\nu_x\in \mathcal{V}_g\}.$$ 
Note that 
 \begin{align*}
\sum_{(x;s)\in A_T(g) } e^{ \int_0^s f(\theta_t(x,0)) \, dt } \le & \sum_{(x;s)\in A_T(g) } e^{ \int_0^s (f+g)(\theta_t(x,0)) \, dt } e^{s(P_\Theta(f)-P_\Theta(f+g)-\beta+\epsilon)}\\
\le & e^{dr}Z_T(f+g)e^{-TP_\Theta(f+g)}e^{T(P_\Theta(f)-\beta+\epsilon)},
\end{align*}
where we used the definition of $\mathcal{V}_g$ and that $|s-T|\le d$. It follows that
 \begin{align*}
\frac{1}{Z_T(f)}\sum_{(x;s)\in A_T(g) } e^{ \int_0^s f(\theta_t(x,0)) \, dt }  \le e^{dr} Z_T(f+g)e^{-TP_\Theta(f+g)}Z_T(f)^{-1}e^{TP_\Theta(f)}e^{T(-\beta+\epsilon)},
\end{align*}
and thus
 \begin{align*}
\frac{1}{Z_T(f)}\sum_{(x;s)\in \bigcup_{i=1}^N A_T(g) } e^{ \int_0^s f(\theta_t(x,0)) dt }  \le&   Ne^{dr} Z_T(f+g)e^{-TP_\Theta(f+g)}Z_T(f)^{-1}e^{TP_\Theta(f)}e^{T(-\beta+\epsilon)}.
\end{align*}
Since $\{{(x,s)\in \mathrm{FPer}_a(T-d,T), \nu_{x}\in \mathcal{K}}\}\subset \bigcup_{i=1}^N A_T(g_i)$, by Lemma \ref{lem:presC0} we obtain that 
 \begin{align*}
\limsup_{T\to\infty}\frac{1}{T}\log\bigg(\frac{1}{Z_T(f)}\sum_{{(x;s)\in \mathrm{FPer}_a(T-d,T), \nu_{x}\in \mathcal{K}}} e^{ \int_0^s f(\theta_t(x,0)) dt }  \bigg)\le -\beta+\epsilon.
\end{align*}
Since $\epsilon>0$ was arbitrary,  the result follows. Finally, the last claim in the statement of Theorem \ref{thm_ldp2} follows directly from standard estimates on exponential growth, as in Remark \ref{rem:positivepressure}.

\subsection{Proof of Theorem \ref{thm_equi2}} 
For the remainder of this section, assume that $f \in \mathrm{C}_b(Y)$ is SPR. Let $\nu_f$ denote the unique equilibrium state of $f$ (see Remark \ref{rem:forder}).

\begin{lemma} \label{lem:beta} If $\nu\in \M_{\le1}(\Theta)$ satisfies that $I_f(\nu)=0$, then $\nu=\nu_f$. 
\end{lemma}

\begin{proof} By assumption, we have that $\int g  \, d \nu \le P_\Theta(f+g)-P_\Theta(f)$, for every $g\in \mathrm{C}_0(Y)$. By Lemma \ref{lem:deri}, it follows that $\int g \, d\nu\le \int g \, d\nu_f$. The same argument applies to $-g$, and therefore $\int g \, d\nu= \int g \, d\nu_f$, for every $g\in \mathrm{C}_0(Y)$. We conclude that $\nu=\nu_f$.
\end{proof}

\begin{remark}\label{rem:obs2} If $\mathcal{K}$ is a closed subset of $\M_{\le1}(\Theta)$ such that $\nu_f\notin \mathcal{K}$, then 
$$\inf \left\{I_f(\nu):\nu\in\mathcal{K} \right\}>0.$$
Indeed, assume for contradiction that $\inf\{I_f(\nu):\nu\in \mathcal{K}\}=0$. Since $I_f$ is lower semicontinuous and $\mathcal{K}$ is compact, there exists $\nu\in \mathcal{K}$ such that $I_f(\nu)=0$. By Lemma \ref{lem:beta}, this implies $\nu_f\in \mathcal{K}$, contradicting the assumption.
\end{remark}

\begin{proof}[Proof of Theorem \ref{thm_equi2}]

Let $g \in C_0(Y)$ and $\epsilon >0$, define
\begin{equation*}
U=U(g, \epsilon)= \left\{\nu \in \M_{\le1}(\Theta) :    \left| \int g \, d \nu - \int g \, d \nu_f   \right| < \epsilon \right\}.
\end{equation*}
By Lemma \ref{prop:test}, in order to prove that $(\nu_T)_T$ converges on cylinder to $\nu_f$, it suffices to show that for  large enough $T$, we have $\nu_T \in U(g, 2 \epsilon)$. Moreover, since for every $T >0$ the measures $\nu_T$ are probabilities and, also, $\nu_f$ is a probability, convergence on cylinders implies weak$^\ast$ convergence (see \cite[Lemma 6.7]{iv}).   Consider the complement of $U$,
\begin{equation*}
\K =\left\{\nu \in \M_{\le1}(\Theta) :    \left| \int g \, d \nu - \int g \, d \nu_f   \right| \geq  \epsilon \right\}.
\end{equation*}
This is a closed subset of $\M_{\le1}(\Theta)$. Since $\nu_f \notin \K$ we have, by Remark \ref{rem:obs2}, that
\begin{equation*}
\inf \left\{I_f(\nu):\nu\in\mathcal{K} \right\}= \beta>0.
\end{equation*}
Let $\beta_0 \in (0, \beta)$, by Theorem \ref{thm_ldp2}  there exists $T_1>0$ such that if $T>T_1$ then
\begin{align} \label{1}
\sum_{(x;s)\in \mathrm{FPer}_a(T-d,T), \nu_{x}\in \mathcal{K}}\exp\left(\int_0^{s}f(\theta_t(x,0)) \,dt\right)
\leq&\\ e^{-\beta_0 T}   \sum_{(x;s)\in \mathrm{FPer}_a(T-d,T)}\exp\left(\int_0^{s}f(\theta_t(x,0))dt\right).
\end{align}
It follows from the definition of pressure and the assumption on the sets $(A_T)_T$ that,
\begin{equation*}
\lim_{T \to \infty} \frac{1}{T} \log \left(	\frac{ \sum_{(x;s)\in A_T}\exp\left(\int_0^{s}f(\theta_t(x,0))dt\right)}{ \sum_{(x;s)\in \mathrm{FPer}_a(T-d,T)}\exp\left(\int_0^{s}f(\theta_t(x,0))dt\right)}	\right)=0.
\end{equation*}
There exists $T_2>0$ such that. if $T>T_2$ then
\begin{align} \label{2}
 \sum_{(x;s)\in \mathrm{FPer}_a(T-d,T)}\exp\left(\int_0^{s}f(\theta_t(x,0))dt\right) \leq e^{\beta_0  T/2} \sum_{(x;s)\in A_T}\exp\left(\int_0^{s}f(\theta_t(x,0))dt\right).
 \end{align}
Combining equations \eqref{1} and \eqref{2}, we obtain, 
\begin{align} \label{3} 
\sum_{(x;s)\in \mathrm{FPer}_a(T-d,T), \nu_{x}\in \mathcal{K}}\exp\left(\int_0^{s}f(\theta_t(x,0)) \,dt\right)
\leq &\\  e^{-\beta_0 T/2}    \sum_{(x;s)\in A_T}\exp\left(\int_0^{s}f(\theta_t(x,0))dt\right).
\end{align} 
By definition we have, 
\begin{equation*}
\mu_T=\frac{1}{\sum_{(x;s)\in A_T}\exp\big(\int_0^{s}f(\theta_t(x,0))dt\big)}\sum_{(x;s)\in A_T}\exp\bigg(\int_0^{s}f(\theta_t(x,0))dt\bigg)\nu_x,
\end{equation*}
thus,
\begin{equation*}
\int g \, d \nu_T = \frac{1}{\sum_{(x;s)\in A_T}\exp\big(\int_0^{s}f(\theta_t(x,0))dt\big)} \sum_{(x;s)\in A_T}\exp\bigg(\int_0^{s}f(\theta_t(x,0))dt\bigg)   \int g \, d \nu_x.
\end{equation*}
It follows from equation \eqref{3} that
\begin{equation} \label{4}
\frac{1}{\sum_{(x;s)\in A_T}\exp\big(\int_0^{s}f(\theta_t(x,0))dt\big)}  \sum_{(x;s)\in A_T, \nu_x \in \K}\exp\bigg(\int_0^{s}f(\theta_t(x,0))dt\bigg)   \int g \, d \nu_x \leq \|g\| e^{-\beta_0 T/2}.
\end{equation}
By the definition of $U$ we have
\begin{eqnarray} \label{5}
\left| 	 \sum_{(x;s)\in A_T, \nu_x \in U}\exp\bigg(\int_0^{s}f(\theta_t(x,0))dt\bigg)  \left( \int g \, d \nu_x  -  \int g \, d \nu_f  \right) 
	\right| \leq &\\  \epsilon \sum_{(x;s)\in A_T, \nu_x \in U}\exp\bigg(\int_0^{s}f(\theta_t(x,0))dt\bigg).
\end{eqnarray}
It follows from  equations \eqref{4} and \eqref{5} that,
\begin{equation*}
\left|	\int g \, d \nu_T - \int g \, d \nu_f	\right| \leq \epsilon +\| g \| e^{-\beta_0 T/2}.
\end{equation*}
Hence, the proof of (a) is complete.

We prove the second assertion of Theorem~\ref{thm_equi2}. Let $\mathcal P_T:=\text{FPer}_a(T-d,T)$
and, for every subset \(E\subset \mathcal P_T\), define its \(f\)-weighted mass by
\[
        W_T(E):=
        \sum_{(x,s)\in E}
        \exp\left(\int_0^s f(\theta_t(x,0))\,dt\right).
\]
Thus, \(W_T(\mathcal P_T)\) is the denominator appearing in the statement of
Theorem~\ref{thm_equi2}(b).

Let $\nu_f$ be the unique equilibrium state of $f$.  Since $M_{\le 1}(\Theta)$ is compact and metrizable for the cylinder topology, we may
choose a decreasing sequence 
$(U_m)_{m\ge1}$ of open neighbourhoods of $\nu_f$ such that
\[
        \bigcap_{m\ge1} U_m=\{\nu_f\}.
\]
Set $K_m:=M_{\le1}(\Theta)\setminus U_m$.
Then $K_m$ is closed and $\nu_f\notin K_m$.  Hence, by Remark~\ref{rem:obs2},
\[
        \beta_m:=\inf_{\nu\in K_m} I_f(\nu)>0 .
\]
By Theorem~\ref{thm_ldp2}, for each fixed $m\ge1$,
\[
 \limsup_{T\to\infty}
 \frac1T
 \log
 \frac{
 W_T\bigl(\{(x,s)\in\mathcal P_T:\nu_x\in K_m\}\bigr)
 }{
 W_T(\mathcal P_T)
 }
 \le -\beta_m .
\]
Consequently, for each $m\ge1$ there exists $R_m>0$ such that, for every
$T\ge R_m$,
\[
        W_T\bigl(\{(x,s)\in\mathcal P_T:\nu_x\in K_m\}\bigr)
        \le e^{-\beta_m T/2} W_T(\mathcal P_T).
\]
Increasing $R_m$, if necessary, we may assume that $R_m\nearrow\infty$ and for every $T\ge R_m$
\[
        e^{-\beta_m T/2}\le \frac1m
       \]
For $T>0$, define
\[
        m(T):=\max\{m\ge1:R_m\le T\},
\]
with $m(T)=1$ if the set on the right-hand side is empty.  Then
$m(T)\to\infty$ as $T\to\infty$.  Define
\[
        A'_T:=
        \{(x,s)\in \mathcal P_T:\nu_x\in U_{m(T)}\}.
\]
Since
\[
        \mathcal P_T\setminus A'_T
        =
        \{(x,s)\in\mathcal P_T:\nu_x\in K_{m(T)}\},
\]
we obtain, for all sufficiently large $T$,
\[
 \frac{W_T(\mathcal P_T\setminus A'_T)}{W_T(\mathcal P_T)}
 \le e^{-\beta_{m(T)}T/2}
 \le \frac1{m(T)}.
\]
Since $m(T)\to\infty$, it follows that
\[
        \frac{W_T(A'_T)}{W_T(\mathcal P_T)}\longrightarrow 1 .
\]

It remains to prove the convergence property.  Let $T_n\to\infty$ and let
$(x_n,s_n)\in A'_{T_n}$. Then
$\nu_{x_n}\in U_{m(T_n)}.$
Since $m(T_n)\to\infty$, for every fixed $m\ge1$ we have
$m(T_n)\ge m$ for all sufficiently large $n$.  Because the sequence
$(U_m)_m$ is decreasing, this implies that, for every fixed $m\ge1$, for all sufficiently large $n \in \N$ we have  $  \nu_{x_n}\in U_m$. 
Therefore, $\nu_{x_n}\to\nu_f$ in the cylinder topology. Since all the measures
$\nu_{x_n}$ and $\nu_f$ are probability measures, cylinder convergence is
equivalent to weak-star convergence. Hence
\[
        \nu_{x_n}\xrightarrow[n\to\infty]{w^*}\nu_f.
\]
This proves Theorem~\ref{thm_equi2}(b) for the window
\(\text{FPer}_a(T-d,T)\).

We also record the cumulative form needed below. Assume, in addition, that
$P_\Theta(f)\ge 0$.  Let
\[
        \mathcal P_{\le T}:=\text{FPer}_a(T)
\]
and define
\[
        W_{\le T}(E):=
        \sum_{(x,s)\in E}
        \exp\left(\int_0^s f(\theta_t(x,0))\,dt\right),
        \qquad E\subset \mathcal P_{\le T}.
\]
For each closed set $K\subset M_{\le1}(\Theta)$ with $\nu_f\notin K$, the
window estimate above implies
\[
 \frac{
 W_{\le T}\bigl(\{(x,s)\in \mathcal P_{\le T}:\nu_x\in K\}\bigr)
 }{
 W_{\le T}(\mathcal P_{\le T})
 }
 \longrightarrow 0 .
\]
Indeed, split $\mathcal P_{\le T}$ into the old part
$\mathcal P_{\le (1-\alpha)T}$ and the remaining part, where
$\alpha\in(0,1)$ is fixed.  Since $P_\Theta(f)>0$, the old part is
exponentially negligible with respect to $W_{\le T}(\mathcal P_{\le T})$.
The remaining part is covered by $O(T)$ windows of length $d$, and on each
window the large deviation estimate gives an exponentially small contribution
from $K$.  The polynomial factor $O(T)$ does not affect the exponential
decay.  Applying the same diagonal construction as above, with
$\mathcal P_{\le T}$ in place of $\mathcal P_T$, yields sets
\[
        A'_{\le T}\subset \text{FPer}_a(T)
\]
such that
\[
        \frac{W_{\le T}(A'_{\le T})}{W_{\le T}(FPer_a(T))}\longrightarrow 1
\]
and such that every sequence $(x_n,s_n)\in A'_{\le T_n}$, with $T_n\to\infty$,
satisfies
\[
        \nu_{x_n}\xrightarrow[n\to\infty]{w^*}\nu_f .
\]
This proves the cumulative version of Theorem~\ref{thm_equi2}(b).
\end{proof}

\section{Boundary points for interval maps} \label{sec:interval}

In this section we derive several applications of our large deviations and equidistribution results for suspension semi-flows to the study of boundary points of interval maps with infinitely many branches. A guiding example throughout is the Gauss map, for which the boundary points are precisely the rational numbers in $(0,1)$. These points correspond to finite continued fraction expansions and hence to finite orbits that terminate at the point where the map is not defined; in this case the point $0$.

A natural difficulty that arises in this setting is that the length of the continued fraction expansion and the size of the derivative along the orbit are, in general, unrelated and therefore induce different orderings of the boundary points. By passing to an associated suspension semi-flow, these quantities become compatible: the length of the flow orbit encodes both the length of the continued fraction expansion and the accumulated expansion of the map. Approximating boundary points by periodic orbits of the suspension, we are then able to apply our general results to obtain quantitative asymptotic statements on their distribution. This approach allows us to relate arithmetic data, such as the size of the denominator of a rational number, to dynamical quantities arising from the flow, thereby providing a framework for understanding the statistical properties of boundary points in interval maps with infinitely many branches.

We now describe the class of maps we will be interested in. Let $(d_n)_n$ be a strictly decreasing sequence of real numbers with $\lim_{n\to\infty} d_n=0$ and $d_1=1$. Set $I_n=[d_{n+1},d_n]$. Consider a map of the unit interval with countably many branches  $T:(0,1] \to [0,1]$  such that each branch $T|_{I_n}$ is monotone and admits a continuous extension to $\overline I_n$ whose image is $[0,1]$. Assume, moreover, that
\begin{enumerate}
\item[(1)] the map is $C^2$ on $\bigcup_{n\geq 1} \textrm{int}(I_n)$,
\item[(2)] there exists $\xi>1$ and $n_0\geq 1$ such that $|(T^{n_0})'(x)|\geq\xi$ for every $x\in \bigcup_{n\geq 1} \textrm{int}(I_n)$, 
\item[(3)] the map satisfies the \emph{Renyi condition}, that is, there exists a positive number $K>0$ such that
\[
\sup_{n\geq 1}\sup_{x,y,z\in I_n} \frac{|T''(x)|}{|T'(y)||T'(z)|}\leq K, and
\]
\item[(4)] there exist constants $c>0$ and $\alpha>1$, such that for every $n\geq 1$
\[
 |T'(x)|\geq c n^\alpha
\qquad
\text{for every }x\in \operatorname{int}(I_n).
\]
\end{enumerate}

An interval map with countable many branches satisfying conditions (1), (2), and (3) is called an expanding–Markov–Renyi (EMR) map. Such maps admit a coding by a full shift on a countable alphabet. If, in addition, condition (4) is satisfied, the map is called \emph{EMR with superlinear branch expansion.} 

Denote by $\B= \bigcup_{n \geq 1} T^{-n}(0)$ the collection of boundary points corresponding to the Markov partition. Since the system is Markov every point not in the boundary has an infinite code. More precisely, if $x\not\in\B$, then we write $x=(a_0,a_1,\dots)$ if and only if $T^n x\in I_{a_n}$ for every $n\geq 0$. On the other hand, points $b \in \B$ have finite codes. Indeed, there exists $\ell(b) \in \N$ such that $T^{\ell(b)}(b)=0$ and $T^{\ell(b)-1}(b)>0$. We say that $\ell(b)$ is the length of the code corresponding to $b$ and we write
\begin{equation*}
b=(a_0, a_1, \dots , a_{\ell(b)-1}),
\end{equation*}
with $a_i \in \N$.

\begin{definition}
For $b \in \B$, consider the measure equidistributed on its finite orbit
\begin{equation*}
\mu_b= \frac{1}{\ell(b)} \sum_{i=0}^{\ell(b)-1} \delta_{T^i b}.
\end{equation*}
\end{definition}
For each $k \in \N$, let
\begin{equation*}
\B_k= \left\{b \in \B :	|(T^{\ell(b)})'(b)| < k 	\right\},
\end{equation*}
where derivatives at boundary points are understood as one-sided derivatives along the corresponding finite branch.

\begin{lemma}\label{lem:Bk-finite}
For every $k\in\mathbb N$, the set
\[
\mathcal B_k=
\left\{
b\in\mathcal B:
\left|(T^{\ell(b)})'(b)\right|<k
\right\}
\]
is finite.
\end{lemma}

\begin{proof}
Throughout the proof, derivatives at boundary points are understood as
one-sided derivatives along the branch determined by the finite code. Let $b\in\mathcal B_k$ and write $m=\ell(b)$. Decompose $m=q n_0+r$ with  $0\le r<n_0$. By the uniform expansion assumption, for each full block of length $n_0$ we have
\[
\left|(T^{n_0})'(T^{j n_0}b)\right|\ge \xi,
\qquad 0\le j\le q-1.
\]
On the remaining $r$ iterates, the polynomial lower bound in the definition of
an EMR map with superlinear branch expansion gives
\[
\left|(T^r)'(T^{q n_0}b)\right|
\ge c^r
\ge c_*,
\]
where $c_*:=\min\{1,c^{n_0-1}\}>0. $ Therefore,
\[
\left|(T^m)'(b)\right|
=
\left|(T^r)'(T^{q n_0}b)\right|
\prod_{j=0}^{q-1}
\left|(T^{n_0})'(T^{j n_0}b)\right|
\ge c_*\,\xi^q.
\]
Since $b\in\mathcal B_k$, we have $\left|(T^m)'(b)\right|<k$. Hence, $c_*\,\xi^q<k$ and so $q$ is bounded above by a constant depending only on $k$. Consequently,
there exists $M_k\in\mathbb N$ such that
\[
\ell(b)\le M_k
\qquad
\text{for every } b\in\mathcal B_k.
\]
We now show that only finitely many symbols can appear in the code of a point
of $\mathcal B_k$. Let
\[
b=(a_0,\ldots,a_{m-1})
\]
be the finite code of $b$, with $m=\ell(b)\le M_k$. By the polynomial lower bound on the derivatives of the branches,
$|T'(T^i b)|\ge c a_i^\alpha$ where $0\le i\le m-1$. Thus
\[
\left|(T^m)'(b)\right|
=
\prod_{i=0}^{m-1}|T'(T^i b)|
\ge
c^m\prod_{i=0}^{m-1}a_i^\alpha.
\]
Set $c_k:=\min\{1,c^{M_k}\}>0.$
Since $m\le M_k$, we have $c^m\ge c_k$. In particular, for every
$0\le i\le m-1$, we have $k> \left|(T^m)'(b)\right| \ge
c_k a_i^\alpha. $ Hence
\[
a_i
<
\left(\frac{k}{c_k}\right)^{1/\alpha}.
\]
Therefore every symbol appearing in the finite code of a point of
$\mathcal B_k$ belongs to the finite set
\[
\left\{
1,\ldots,
\left\lfloor
\left(\frac{k}{c_k}\right)^{1/\alpha}
\right\rfloor
\right\}.
\]

We have shown that every point in $\mathcal B_k$ has a finite code whose
length is at most $M_k$ and whose symbols belong to a finite alphabet depending
only on $k$. Hence there are only finitely many possible finite codes. 
\end{proof}

\subsection{Derivative and length of finite orbits} 
A difficulty that must be addressed is that the length of the coding of a boundary point, denoted by $\ell(b)$, need not be related to the size of its derivative $|(T^{\ell(b)})'(b)|$. However, the expansivity condition (2) for the map $T$ implies that both quantities grow to infinity if the coding length tends to infinity. In Proposition~\ref{thm:weakdixonIM} we show that the length of the coding tends to infinity for a very large proportion of boundary points. 

In the setting of continued fractions, these questions have been extensively studied. In this context, boundary points for the Gauss map correspond to rational numbers $p/q$, the size of the derivative is given by $2\log q$, and the length of the coding coincides with the number of steps required by the Euclidean algorithm to compute the greatest common divisor of $p$ and $q$. For example, Heilbronn~\cite{he}, and later Dixon~\cite{di}, proved that for a large proportion of rational numbers, as $q \to \infty$, the coding length $\ell(p/q)$ is proportional to $\log q$.

\begin{definition} We say that an EMR interval map $T$ satisfies the $D-$condition, if for every $k \in \N$ there exists subsets $A_k \subset \B_k$  with
\begin{equation*}
\lim_{k \to \infty} \frac{\# A_k}{\# \B_k}=1,
\end{equation*}
with the property that for every sequence $(b_k)_k$ with  $b_k \in A_k$ we have $\lim_{k \to \infty} \ell(b_k)=\infty$.
\end{definition}

\begin{proposition}\label{thm:weakdixonIM} An EMR interval map with superlinear branch expansion satisfies the $D-$condition.
\end{proposition}

The proof of this proposition relies on the fact that the proportion of boundary points whose coding length is bounded by a given constant is asymptotically negligible compared to the growth of their derivatives. To simplify notation, consider for every $b=(a_0, \dots , a_{\ell(b)-1})$ the \emph{logarithm of the multiplier} $\lambda(b)$ of $b$, defined by $\lambda(b)= \log |(T^{\ell(b)})'|(b)$.

\begin{proposition}\label{prop:largelengt}
For every $N\ge 1$, we have
\[
\lim_{n\to\infty}
\frac{
\#\{b\in \mathcal{B}:\lambda(b)\le n,\ \ell(b)\le N\}
}{
\#\{b\in \mathcal{B}:\lambda(b)\le n\}
}=0 .
\]
\end{proposition}
In order to prove this proposition we will require the following estimate.

\begin{lemma}\label{lem:divisor-estimate}
For every integer $j\ge 1$ and every $t\ge 1$, we have
\[
\#\left\{(a_0,\ldots,a_{j-1})\in\mathbb N^j:
\prod_{i=0}^{j-1}a_i\le t\right\}
\le t(1+\log t)^{j-1} .
\]
\end{lemma}

\begin{proof}
Let $D_j(t)$ denote the cardinality on the left-hand side. For $j=1$,
$D_1(t)=\lfloor t\rfloor\le t$. Suppose now that $j\ge 2$ and that the
estimate has been proved for $j-1$. Since $t/a\ge 1$ whenever
$1\le a\le \lfloor t\rfloor$, the induction hypothesis gives
\[
\begin{aligned}
D_j(t)
&=\sum_{a=1}^{\lfloor t\rfloor}D_{j-1}(t/a) \\
&\le t\sum_{a=1}^{\lfloor t\rfloor}\frac{1}{a}
   \bigl(1+\log(t/a)\bigr)^{j-2} \\
&\le t(1+\log t)^{j-2}\sum_{a=1}^{\lfloor t\rfloor}\frac{1}{a} \\
&\le t(1+\log t)^{j-1}.
\end{aligned}
\]
Here we used the elementary bound
$\sum_{a=1}^{\lfloor t\rfloor}a^{-1}\le 1+\log t$. The induction is
complete.
\end{proof}

\begin{proof}[Proof of Proposition~\ref{prop:largelengt}]
Let $\widehat\tau:\Sigma\to\mathbb R$ be the symbolic potential $\widehat\tau(x)=\log |T'(\pi x)|,$
where $\pi:\Sigma\to[0,1]$ denotes the coding map. For a finite word
$w=(a_0,\ldots,a_{m-1})$, let $b_w\in\mathcal B$ be the boundary point
with code $w$, and let
$x_w=(a_0,\ldots,a_{m-1},a_0,\ldots,a_{m-1},\ldots)$
be the corresponding periodic point in $\Sigma$. We write $|w|=m$ and $\lambda(w):=\lambda(b_w)=\log |(T^m)'(b_w)|.$

We first estimate the numerator. Fix $1\le j\le N$ and assume that
$|w|=j$ and $\lambda(w)\le n$. By the polynomial lower bound in the definition of an EMR map with superlinear branch expansion, for every $ 0\le i\le j-1$ we have $|T'(T^i b_w)|\ge c a_i^\alpha$.
Therefore,
\[
n\ge \lambda(w)
=\sum_{i=0}^{j-1}\log |T'(T^i b_w)|
\ge j\log c+\alpha\sum_{i=0}^{j-1}\log a_i .
\]
It follows that
\[
\prod_{i=0}^{j-1}a_i
\le c^{-j/\alpha}e^{n/\alpha}
\le C_Ne^{n/\alpha},
\]
where $C_N:=\max\{1,c^{-N/\alpha}\}$. Hence
\[
\#\{b\in\mathcal B:\lambda(b)\le n,\ \ell(b)=j\}
\le
\#\left\{(a_0,\ldots,a_{j-1})\in\mathbb N^j:
\prod_{i=0}^{j-1}a_i\le C_Ne^{n/\alpha}\right\}.
\]
By Lemma~\ref{lem:divisor-estimate}, applied with $t=C_Ne^{n/\alpha}$, we obtain
\[
\begin{aligned}
\#\{b\in\mathcal B:\lambda(b)\le n,\ \ell(b)=j\}
&\le  C_N e^{n/\alpha}
   \left(1+\log C_N+\frac{n}{\alpha}\right)^{j-1} \\
&\le  C_N e^{n/\alpha}
   \left(|1+\log C_N|+\frac{n}{\alpha}\right)^{j-1} \\
&\le  C_N e^{n/\alpha}
   \left(|1+\log C_N|+\frac{1}{\alpha}\right)^{j-1}(1+n)^{j-1}
\\
&\le  C'_{j,N} e^{n/\alpha}(1+n)^{j-1}.
\end{aligned}
\]
where $C'_{j,N}= C_N \left(|1+\log C_N|+\frac{1}{\alpha}\right)^{j-1}.$ After summing over $1\le j\le N$, we obtain
\begin{equation}\label{eq:fixed-length-upper}
\#\{b\in\mathcal B:\lambda(b)\le n,\ \ell(b)\le N\}
\le Q_N(n)e^{n/\alpha},
\end{equation}
where $Q_N(n)=\sum_{j=1}^N C'_{j,N}(1+n)^{j-1}$ is a polynomial depending only on $N$.

We now estimate the denominator from below. Set
\begin{equation} \label{sym-tau}
\tau:=\frac1{n_0}\sum_{q=0}^{n_0-1}\widehat\tau\circ\sigma^q .
\end{equation}
Then $\tau$ is cohomologous to $\widehat\tau$, has summable variations (\cite[Proposition~3.1.3]{j}), and is
bounded away from zero; indeed, the uniform expansion assumption gives
$\tau\ge n_0^{-1}\log\xi>0$. Moreover, if $x$ is periodic with period $m$,
then $S_m\tau(x)=S_m\widehat\tau(x).$ Indeed,
\[
\begin{aligned}
S_m\tau(x_w)
&=
\frac1{n_0}\sum_{q=0}^{n_0-1}
\sum_{i=0}^{m-1}
\widehat\tau(\sigma^{i+q}x_w)  
&=
\frac1{n_0}\sum_{q=0}^{n_0-1}
S_m\widehat\tau(x_w)
=
S_m\widehat\tau(x_w),
\end{aligned}
\]
because shifting the starting point of the sum along an $m$-periodic orbit
does not change the sum over one full period. Since $\widehat\tau$ has summable variations, the constant
$C_{\mathrm{var}}:=\sum_{q=1}^{\infty}V_q(\widehat\tau)$
is finite. For a word $w$ with $|w|=m$, the points $T^i b_w$ and
$\pi(\sigma^i x_w)$ have the same first $m-i$ symbols. Hence
\[
\begin{aligned}
\left|\lambda(w)-S_m\tau(x_w)\right|
=
\left|\lambda(w)-S_m\widehat\tau(x_w)\right| 
\le \sum_{i=0}^{m-1} V_{m-i}(\widehat\tau) 
\le C_{\mathrm{var}} .
\end{aligned}
\]
Consequently, if $S_{|w|}\tau(x_w)\le n-C_{\mathrm{var}}$, then
$\lambda(w)\le n$. Let
\[
\mathcal N(R):=
\#\{w:S_{|w|}\tau(x_w)\le R\}.
\]
The preceding implication and the correspondence between finite words and boundary points give
\begin{equation}\label{eq:denominator-lower-by-N}
\#\{b\in\mathcal B:\lambda(b)\le n\}
\ge \mathcal N(n-C_{\mathrm{var}})
\end{equation}
for all sufficiently large $n$. We will use now the fact that the suspension semi-flow with roof function $\tau$ has topological entropy equal to $1$ (see (1) in Proposition \ref{en_flo}). Hence, for any fixed symbol $a$ and any $d>0$, Theorem \ref{thm: flow pres} ensures that
\[
\lim_{R\to\infty}
\frac1R\log \#\mathrm{FPer}_a(R-d,R)=1.
\]
Every element of $\mathrm{FPer}_a(R-d,R)$ determines a word counted by
$\mathcal N(R)$. Therefore
\[
\liminf_{R\to\infty}\frac1R\log\mathcal N(R)\ge 1.
\]
Thus, for every $\varepsilon>0$ and all sufficiently large $R$, we have $\mathcal N(R)\ge e^{(1-\varepsilon)R}.$
Taking $R=n-C_{\mathrm{var}}$ in \eqref{eq:denominator-lower-by-N}, we get
\begin{equation}\label{eq:denominator-lower}
\#\{b\in\mathcal B:\lambda(b)\le n\}
\ge e^{(1-\varepsilon)(n-C_{\mathrm{var}})}
\end{equation}
for all sufficiently large $n$. Choose $\varepsilon>0$ such that $1-\varepsilon>1/\alpha$, which is possible
because $\alpha>1$. Combining \eqref{eq:fixed-length-upper} and
\eqref{eq:denominator-lower}, we obtain
\[
\frac{
\#\{b\in\mathcal B:\lambda(b)\le n,\ \ell(b)\le N\}
}{
\#\{b\in\mathcal B:\lambda(b)\le n\}
}
\le
C_\varepsilon Q_N(n)
\exp\left(n\left(\frac1{\alpha}-(1-\varepsilon)\right)\right).
\]
The exponent is negative, so the right-hand side tends to zero as
$n\to\infty$. This proves the proposition.
\end{proof}

\begin{proof}[Proof of Proposition \ref{thm:weakdixonIM}] Let $k\in\N$. For any $N\geq 1$ consider the set $\B_{k,N}\subset \B_k$ defined as
\[
\B_{k,N}=\{b\in \B_k : \ell(b)\leq N\}. 
\]
By Proposition \ref{prop:largelengt}, we know that $\# \B_{k,N}/\# \B_k \to 0$ as $k\to\infty$. In particular, there exists $k_N\geq 1$ such that for every $k\geq k_N$, we have
\[
\frac{\#\{b\in \B_k : \ell(b)\geq N\}}{\# \B_k} \geq 1-\frac{1}{N}.
\]
For every $k\in [k_N,k_{N+1})$, set $A_k=\{b\in \B_k : \ell(b)\geq N\}$. Then, by construction $\# A_k/\# \B_k \to 1$ as $k\to\infty$. Moreover, by construction $\ell(b_k)\to\infty$ for every $b_k\in A_k$ as $k\to\infty$.
\end{proof}

\subsection{Proof of Theorem \ref{main:interval}}

We will start by showing that, from a dynamical point of view, a sequence of boundary points behaves in a similar way to a sequence of periodic orbits. Let $(b_k)_k$ be a sequence of boundary points with $b_k\in \B_k$. Denote by $\omega_k=(a_0,\ldots,a_{\ell(b_k)-1})$ the word representing $b_k$. Set $\mu_k:=\mu_{b_k}$ and define $\hat{\mu}_k$ as the atomic measure uniformly distributed on the periodic orbits $x_{\omega_k}:=x_k=[\overline{a_0,\ldots,a_{\ell(b_k)-1}}]$, that is
\[
\hat{\mu}_k := \frac{1}{\ell(b_k)}\left( \delta_{[\overline{a_0,\ldots,a_{\ell(b_k)-1}}]}+\delta_{[\overline{a_1,\ldots,a_{\ell(b_k)-1},a_0}]}+\ldots+ \delta_{[\overline{a_{\ell(b_k)-1},\ldots,a_{\ell(b_k)-2}}]}\right).
\]

\begin{proposition}\label{prop:per_and_div} If $\lim_{k\to\infty} \ell(b_k) =\infty$, then $\lim_{k\to\infty} d(\mu_k,\hat{\mu}_k)=0$.
\end{proposition}
\begin{proof}
First note that, since for every $i \in \{0, \dots, \ell(b_k)-1\}$ we have that both $T^i(b_k)$ and $T^i x_k$ belong to the cylinder
$I_{[a_i, \dots , a_{l_\ell(b_k)-1}]}$, we have that (see \cite[Page 172]{cfs})
\begin{equation} \label{separacion}
|T^{i}(b_k)) - T^i(x_k) | \leq \frac{1}{\xi^{[(\ell(b_k)-i)/n_0]}},
\end{equation}
where $[\cdot]$ denotes floor function and $\xi$ is the uniform expansion assumption $(2)$. 

In order to prove $\lim_{k \to \infty} d(\mu_k, \hat{\mu}_k)=0$ we just need to prove (see \cite[Theorem 13.16 (ii)]{kl}) that for every Lipschitz function $f:[0,1] \to \R$,
\begin{equation*}
\lim_{n \to \infty} \left| \int f \, d \mu_k - \int f \, d \hat{\mu}_k \right| =0.
\end{equation*} 
Let $f:[0,1] \to \R$ be a Lipschitz function with Lipschitz constant equal to $L$. We have, by inequality \eqref{separacion}, that

\begin{eqnarray*}
\left| \int f \, d \mu_k - \int f \, d \hat{\mu}_k \right| &\leq& \frac{1}{\ell(b_k)} \sum_{i=0}^{\ell(b_k)-1} |	f(	T^{i}(	b_k)) -	f(	  T^i (x_k))|\\ 
&\leq& \frac{L}{\ell(b_k)} \sum_{i=0}^{\ell(b_k)-1}|T^{i}(	b_k) -	  T^i (x_k)| \\ 
&\leq & \frac{L}{\ell(b_k)} \frac{n_0\xi}{\xi-1}, 
\end{eqnarray*}
where the last inequality follows from \eqref{separacion} and comparison with a geometric series. The result now follows since $\lim_{k \to \infty} \ell(b_k)= \infty$.
\end{proof}

In order to prove Theorem \ref{main:interval}, we will use Theorem \ref{thm_equi2} for the suspension semi-flow over the full-shift $(\Sigma,\sigma)$ induced by $T$ with a suitable roof function. As observed in Remark \ref{rmk:bipPer}, in the case of the full-shift we may consider $\textrm{FPer}(T-d,T)$ instead of $\textrm{FPer}_a(T-d,T)$ for the remainder of this section.

\begin{proposition} \label{en_flo}
Let \(\Theta=(\theta_t)_{t\geq 0}\) be the suspension semi-flow over
\((\Sigma,\sigma)\) with roof function \(\tau\). Then
\begin{enumerate}
\item[(1)] the entropy satisfies \(h_{top}(\Theta)=1\);
\item[(2)] the flow is SPR, that is, $h_{\infty}(\Theta)<h_{top}(\Theta).$
\end{enumerate}
\end{proposition}

\begin{proof} If $x\in\Sigma$ is $n$-periodic, then
\[
\begin{aligned}
S_n\tau(x)
&=
\sum_{i=0}^{n-1}\tau(\sigma^i x)  
&=
\frac1{n_0}
\sum_{q=0}^{n_0-1}
\sum_{i=0}^{n-1}
\widehat\tau(\sigma^{i+q}x).
\end{aligned}
\]
Since \(x\) is \(n\)-periodic, for every fixed \(q\),
\[
\sum_{i=0}^{n-1}\widehat\tau(\sigma^{i+q}x)
=
\sum_{i=0}^{n-1}\widehat\tau(\sigma^i x)
=
S_n\widehat\tau(x).
\]
Therefore $S_n\tau(x)=S_n\widehat\tau(x)$
for every periodic point $x$. Hence, the periodic partition sums defining the Gurevich pressure coincide, and so $P(-t\tau)=P(-t\widehat\tau)$
for every $t\in\R$ for which the pressures are defined.

By \cite[Theorem~4.2.13]{mubook}, applied to the interval system and then
written in symbolic coordinates through the coding map $ \pi$, we have that the Hausdorff dimension, $\operatorname{HD}$ of $\Lambda$ satisfies, 
\[
\inf\left\{
t\in\mathbb R:
P_\sigma(-t\widehat\tau)\le 0
\right\}
=
\operatorname{HD}(\Lambda).
\]
Since \(\Lambda=[0,1]\setminus\mathcal B\) and \(\mathcal B\) is countable, $\operatorname{HD}(\Lambda)=1.$
Using the equality $P_\sigma(-t\tau)=P_\sigma(-t\widehat\tau),$
we obtain
$\inf\left\{t\in\mathbb R:P(-t\tau)\le 0\right\}=1.$
Therefore, by Theorem~\ref{thm: flow pres} applied to the potential \(f=0\),
\[
h_{top}(\Theta) =\inf\left\{t\in\mathbb R:P_\sigma(-t\tau)\le 0\right\}=1.
\]
This proves claim $(1)$. We now prove that the flow is SPR. Let $t>1/\alpha$. We show that $P(-t\tau)<\infty.$
By the pressure comparison above, it is enough to prove that $P(-t\widehat\tau)<\infty$. Let $x=(a_0,a_1,\ldots)\in\Sigma$. Since $\pi(\sigma^i x)\in I_{a_i},$ condition $(4)$ in the definition of an EMR map with superlinear branch expansion gives $|T'(\pi(\sigma^i x))| \ge c a_i^{\alpha}.$
Equivalently,
\[
\widehat\tau(\sigma^i x)
=
\log |T'(\pi(\sigma^i x))|
\ge
\log c+\alpha\log a_i .
\]
Hence, for every \(n\)-periodic point \(x=(a_0,a_1,\ldots)\),
\[
\exp\left(-tS_n\widehat\tau(x)\right)
\le
c^{-tn}
\prod_{i=0}^{n-1}a_i^{-t\alpha}.
\]
Using the periodic-orbit formula for the Gurevich pressure on the full shift,
we obtain
\[
\begin{aligned}
P(-t\widehat\tau)
&\le
\limsup_{n\to\infty}
\frac1n
\log
\sum_{a_0,\ldots,a_{n-1}}
c^{-tn}
\prod_{i=0}^{n-1}a_i^{-t\alpha}  \\
&=
-t\log c
+
\limsup_{n\to\infty}
\frac1n
\log
\left(
\sum_{k=1}^{\infty}k^{-t\alpha}
\right)^n  \\
&=
-t\log c
+
\log
\left(
\sum_{k=1}^{\infty}k^{-t\alpha}
\right).
\end{aligned}
\]
Since $t>1/\alpha$, the series
$\sum_{k=1}^{\infty}k^{-t\alpha}$
converges. Therefore
$P(-t\widehat\tau)<\infty,$
and consequently $P(-t\tau)<\infty$, for every $t>1/\alpha$.
It follows that
\[
s_\infty(-\tau)
:=
\inf\left\{
t\in\mathbb R:
P_\sigma(-t\tau)<\infty
\right\}
\le
\frac1{\alpha}.
\]
By Lemma~\ref{lem:entropyinf}, $h_\infty(\Theta)=s_\infty(-\tau).$ Since $\alpha>1$, we have
\[
h_\infty(\Theta)
\le
\frac1{\alpha}
<
1
=
h_{top}(\Theta).
\]
Therefore, $\Theta$ is SPR. This proves claim $(2)$.
\end{proof}

\begin{proof}[Proof of Theorem \ref{main:interval}]
Let $T$ be an EMR interval map with superlinear branch expansion and let $(\Sigma,\sigma)$ be the full shift
coding of $T$. Let $\widehat\tau:\Sigma\to\mathbb R$ be the symbolic potential $\widehat\tau(x)=\log |T'(\pi x)|,$
and $\tau:\Sigma \to \R$ as defined in equation \eqref{sym-tau}. Recall that functions $\tau$ and $\widehat\tau$ are cohomologous, and $\tau$ has summable variations and is
bounded away from zero.  Let $w=(a_0,\ldots,a_{m-1})$ be a finite word. As before,  let $b_w\in B$ be the boundary point with code $w$, and let
$x_w\in\Sigma$ be the periodic sequence $x_w=(a_0,\ldots,a_{m-1},a_0,\ldots,a_{m-1},\ldots).$
Notice that $|w|=m$ is a period of $x_w$, although it need not be
the minimal period. Since $x_w$ is $m$-periodic, recall  that we have
\begin{equation}\label{eq:tau-hattau-periodic}
        S_m\tau(x)=S_m\widehat\tau(x).
\end{equation}
Set  $s(w):=S_m\tau(x_w)$ and $\lambda(w):=\log |(T^m)'(b_w)|$.  We first record a uniform comparison between these two quantities. Since
$\widehat\tau$ has summable variations, there exists $C_0>0$ such that for every finite word $w$ we have,
\begin{equation}\label{eq:lambda-s-comparison}
        |\lambda(w)-s(w)|\le C_0
   \end{equation}
Indeed, by the equality in equation \eqref{eq:tau-hattau-periodic}, it is enough to compare
$\lambda(w)$ with $S_m\widehat\tau(x_w)$. For $0\le j\le m-1$, the interval
point $T^j b_w$ and the symbolic point $\pi(\sigma^j x_w)$ belong to the same
cylinder of length $m-j$. Therefore,
\[
\begin{aligned}
 |\lambda(w)-S_m\widehat\tau(x_w)|
 &\le
 \sum_{j=0}^{m-1}
 \left|
      \log |T'(T^j b_w)|
      -\widehat\tau(\sigma^j x_w)
 \right| 
 &\le
 \sum_{j=0}^{m-1} V_{m-j}(\widehat\tau)
 \le
 \sum_{r=1}^{\infty} V_r(\widehat\tau)<\infty .
\end{aligned}
\]
This proves the claim in equation \eqref{eq:lambda-s-comparison}.    Let $\mathcal W(R):=\{w:s(w)\le R\}$ and $N(R):=\#\mathcal W(R).$
By Proposition~\ref{en_flo}, the suspension semi-flow over $(\Sigma,\sigma)$
with roof function $\tau$ has topological entropy equal to $1$ and is SPR.
Hence the large-deviation estimate of Theorem~\ref{thm_ldp2}, applied with
$f=0$, implies the following cumulative density-one statement for marked
periodic words. There exist sets
\[
        \mathcal G_k\subset \mathcal W(\log k+C_0)
\]
such that
\begin{equation}\label{eq:good-sets-density}
        \frac{\#(\mathcal W(\log k+C_0)\setminus \mathcal G_k)}
        {N(\log k-C_0)}\longrightarrow 0,
\end{equation}
and such that, if $w_k\in\mathcal G_k$, then the invariant probability measure
$\nu_{w_k}$ supported on the corresponding periodic orbit of the suspension
semi-flow satisfies
\begin{equation}\label{eq:flow-periodic-conv}
        \nu_{w_k}\xrightarrow[k\to\infty]{}\nu_{\mathrm{mme}}
\end{equation}
in the cylinder topology, hence in the weak$^\ast$ topology. Here
$\nu_{\mathrm{mme}}$ denotes the measure of maximal entropy of the suspension
semi-flow. Indeed, for each neighbourhood
$U$ of $\nu_{\mathrm{mme}}$, Theorem~\ref{thm_ldp2} and the fact that
$\nu_{\mathrm{mme}}$ is the unique zero of the rate function imply that the
proportion of periodic words with $s(w)\le R$ and $\nu_w\notin U$ decays
exponentially. Since $N(R)$ has exponential growth rate $1$, the same estimate
holds after replacing $N(R)$ in the denominator by $N(R-2C_0)$. Taking a
countable neighbourhood basis of $\nu_{\mathrm{mme}}$ and using a standard
diagonal argument gives the sets $\mathcal G_k$. The same construction also
removes the exponentially negligible set of words with $s(w)\le \eta\log k$,
for some fixed $\eta\in(0,1)$; hence every sequence $w_k\in\mathcal G_k$
satisfies $s(w_k)\to\infty$.  We now define
\[
        A_k:=\{b_w\in\mathcal B_k:w\in\mathcal G_k\}.
\]
By definition, $A_k\subset\mathcal B_k$. We claim that
\[
        \frac{\#A_k}{\#\mathcal B_k}\longrightarrow 1.
\]
Indeed, the bound in equation \eqref{eq:lambda-s-comparison} gives the inclusions $\mathcal W(\log k-C_0)\subset \mathcal B_k
        \subset \mathcal W(\log k+C_0), $ where we identify a word with its corresponding boundary point. Therefore,
\[
\begin{aligned}
 \frac{\#(\mathcal B_k\setminus A_k)}{\#\mathcal B_k}
 &\le
 \frac{\#(\mathcal W(\log k+C_0)\setminus \mathcal G_k)}
 {N(\log k-C_0)}.
\end{aligned}
\]
The right-hand side tends to zero by \eqref{eq:good-sets-density}. Thus
$\#A_k/\#\mathcal B_k\to1$. Let now $b_{w_k}\in A_k$ be arbitrary. Since $w_k\in\mathcal G_k$, the convergence in equation 
\eqref{eq:flow-periodic-conv} holds. Let
\[
        \widehat m_{w_k}
        :=\frac1{|w_k|}\sum_{j=0}^{|w_k|-1}\delta_{\sigma^j x_{w_k}}
\]
be the corresponding symbolic periodic measure. By the Ambrose--Kakutani
correspondence, and by Remark~\ref{rem:convprob}, we have that  equation  \eqref{eq:flow-periodic-conv}
implies
\begin{equation}\label{eq:base-periodic-conv}
        \widehat m_{w_k}\xrightarrow[k\to\infty]{}m_1
        \qquad\text{and}\qquad
        \int \tau\,d\widehat m_{w_k}	\xrightarrow[k\to\infty]{} \int \tau\,dm_1,
\end{equation}
where $m_1$ is the equilibrium state of the symbolic potential $-\tau$. Since
$\tau$ and $\widehat\tau$ are cohomologous, $m_1$ is also the equilibrium state
of $-\widehat\tau$. Its projection $\pi_*m_1$ is the equilibrium measure
$\mu_1$ of the interval potential $-\log |T'|$. Moreover,
\[
        \int \tau\,d\widehat m_{w_k}=\frac{s(w_k)}{|w_k|}.
\]
Since $s(w_k)\to\infty$ and the right-hand side converges to the finite number
$\int\tau\,dm_1$, we have $|w_k|\to\infty$.

We next compare the finite boundary orbit with the periodic orbit determined
by $w_k$. Define
\[  \mu_{b_{w_k}}
        :=\frac1{|w_k|}\sum_{j=0}^{|w_k|-1}\delta_{T^j b_{w_k}}
\quad
\text{ and }
\quad
        \widetilde\mu_{w_k}:=\pi_*\widehat m_{w_k}
        =\frac1{|w_k|}\sum_{j=0}^{|w_k|-1}
          \delta_{\pi(\sigma^j x_{w_k})}.
\]
The expansion estimate for EMR maps gives a constant $C_1>0$ such that
\[
        |T^j b_{w_k}-\pi(\sigma^j x_{w_k})|
        \le C_1\xi^{-\lfloor(|w_k|-j)/n_0\rfloor}
\]
for every $0\le j\le |w_k|-1$. Hence, for every Lipschitz function
$\varphi:[0,1]\to\mathbb R$,
\[
\begin{aligned}
 \left|\int\varphi\,d\mu_{b_{w_k}}
       -\int\varphi\,d\widetilde\mu_{w_k}\right|
 &\le
 \frac{\operatorname{Lip}(\varphi)}{|w_k|}
 \sum_{j=0}^{|w_k|-1}
 |T^j b_{w_k}-\pi(\sigma^j x_{w_k})|  
 &\le
 \frac{C_2\operatorname{Lip}(\varphi)}{|w_k|}
 \longrightarrow 0 .
\end{aligned}
\]
Since $\widetilde\mu_{w_k}=\pi_*\widehat m_{w_k}$ and
$\widehat m_{w_k}\to m_1$, we have
$\widetilde\mu_{w_k}\to\pi_*m_1=\mu_1$. Therefore
\[
        \mu_{b_{w_k}}\xrightarrow[k\to\infty]{w^\ast}\mu_1.
\]
This proves the first assertion. It remains to prove the convergence of the derivative averages. By the bound in equation
\eqref{eq:lambda-s-comparison},
\[
        \left|
        \frac{\lambda(w_k)}{|w_k|}-\frac{s(w_k)}{|w_k|}
        \right|
        \le \frac{C_0}{|w_k|}\longrightarrow 0.
\]
Using the convergence in equation \eqref{eq:base-periodic-conv}, we obtain
\[
        \frac{s(w_k)}{|w_k|}
        =\int\tau\,d\widehat m_{w_k}
        \longrightarrow
        \int\tau\,dm_1.
\]
Since $\tau$ is cohomologous to $\widehat\tau$ and
$\widehat\tau=\log |T'|\circ\pi$, we have
\[
        \int\tau\,dm_1
        =\int\widehat\tau\,dm_1
        =\int\log |T'|\,d\mu_1.
\]
Finally, $|w_k|=\ell(b_{w_k})$ and  $\lambda(w_k)=\log |(T^{\ell(b_{w_k})})'(b_{w_k})|.$
Thus
\[
        \lim_{k\to\infty}
        \frac1{\ell(b_{w_k})}
        \log |(T^{\ell(b_{w_k})})'(b_{w_k})|
        =
        \int\log |T'|\,d\mu_1.
\]
This proves the second assertion and completes the proof.
\end{proof}

\subsection{Equidistribution of quadratic irrationals} 
A \emph{quadratic irrational} is an irrational real number satisfying a quadratic equation $ax^2+bx+c=0$, with $a,b,c \in \Z$ and $a \neq 0$.  Lagrange proved in 1770 that an irrational number has an eventually periodic continued fraction expansion if and only if it is a quadratic
irrational.
It has long been known that that purely periodic quadratic irrationals $x=[\overline{a_1, \dotsm a_n}]$ equidistribute with respect to the Gauss measure when ordered by the corresponding Lyapunov exponent $\frac{1}{n}\sum_{i=0} \log|(G^n(x)'|$, see for example \cite[Section 5]{p1} or \cite[Section 8]{m}. This equidistribution result follows directly from Theorem~\ref{thm_equi2} by considering the suspension flow over the Gauss map with roof function $\tau=\log |G'|$, since, as observed in the previous section, its symbolic model satisfies the required assumptions. More generally, for every expanding-Markov-Renyi maps the  periodic points equidistribute  with respect to the equilibrium measure of $-\log|T'|$ holds when ordered by its Lyapuniv exponent.

\subsection{Haas-Molnar continued fractions} \label{sec:examples} 
In this section we present examples of interval maps to which Theorem~\ref{main:interval} applies. In particular, we study a generalized class of continued fraction maps introduced by Haas and Molnar~\cite{hm}. Fix $r \in \N$ and consider the map
\[
T_r \colon (0,1] \to [0,1], \qquad
T_r(x) = \frac{r}{x} - \Big[ \frac{r}{x} \Big],
\]
where $[ x ]$ denotes the integer part of $x$. Note that $T_r$ satisfies the assumptions of Section \ref{sec:interval} with  $(d_n^r)_{n\ge r}$ given by $d_n^r =r/n$, for every $n \geq r$. Moreover,  $T_r$ admits a unique equilibrium measure for the potential $-\log |T_r'|$  (see \cite{hm} for its construction), namely
\[
\mu_1^r(A)
=
\frac{1}{\log\!\left(\frac{r+1}{r}\right)}
\int_A \frac{1}{x+r}\, dx .
\]
Associated with $T_r$ there is a continued fraction expansion of the form:
\[
x
=
\cfrac{r}{a_1 + r + \cfrac{r}{a_2 +r + \cfrac{r}{a_3 + r + \dots}}}
=
[a_1,a_2,a_3,\dots]_r,
\]
where $a_i \in \mathbb{N} \cup \{0\}$. In this representation, the map $T_r$ acts as the full-shift. The corresponding continued fraction theory was developed in~\cite{hm}. Note that when $r=1$ the map $T_r$ is the Gauss map, $\mu_1^1$  is the Gauss measure  and we obtain the classical continued fraction theory \cite{hw,ki}.

The set of boundary points  $\B= \bigcup_{n \geq 1} T_r^{-n}(0)$ will be denoted by $\mathbb{Q}_r$ and called the set of 
 \emph{$r$-rational points}. Note that $\mathbb{Q}_r \subset \mathbb{Q}$ and that if $r=1$ then $\mathbb{Q}_r=\mathbb{Q}$. 
 
\begin{remark}
If  $p/q \in \mathbb{Q}_r$ is such that $p/q=[a_1, \dots , a_{\ell(p/q)}]$ then
\begin{equation*}
|(T^{\ell(p/q)}(p/q))'|= r^{\ell(p/q) -2} q^2.
\end{equation*}
In particular, the logarithm of the multiplier of $p/q$ is given by 
\begin{equation} \label{r-order}
\lambda(p/q)= \sum_{i=0}^{\ell(p/q)-1} \log|T_r'(T^i(p/q))| = 2\log q + (\ell(p/q) -2) \log r.
\end{equation}
For $r=1$ this is a classical result. As observed in Section \ref{sec:interval}, ordering rationals (or $r$-rationals) by the size of their denominator and by the length of their orbit leads to different orderings. Remarkably, by equation~\eqref{r-order}, when considering the suspension flow over $T_r$ with roof function $\log|T'_r|$, the orbit corresponding to $(p/q, 0)$ has length $ 2\log q + (\ell(p/q) -2) \log r$ and these two orderings become compatible. In the case $r=1$ this correspondence is particularly transparent.
\end{remark}

Theorem~\ref{main:interval} implies that the orbits of boundary points, $\mathbb{Q}_r$, equidistribute with respect to the equilibrium  measure $\mu_1^r$.  Indeed, for $b \in \mathbb{Q}_r$ let
\begin{equation*}
\mu_b= \frac{1}{\ell(b)} \sum_{i=0}^{\ell(b)-1} \delta_{T_r^i b}.
\end{equation*}
For each $k \in \N$ let
\begin{equation*}
\B_k= \left\{b \in \mathbb{Q}_r :	|(T_r^{\ell(b)})'(b)| < k 	\right\}.
\end{equation*}
We have the following result:

\begin{theorem} \label{thm-cont-frac}
  For every $k\in\N$ there exists $A_k\subset \B_k$ such that
\[
\lim_{k\to\infty} \frac{\# A_k}{\# \B_k} = 1,
\]
with the property that for every $k\in \N$ and $b_k\in A_k$,  the sequence of measures $(\mu_{b_k})_k$ converges in the weak$^\ast$ topology, as $k$ tends to infinity, to the equilibrium measure $\mu_1^r$ of $-\log|T_r'|$. Moreover, 
\[
\lim_{k\to\infty}
\frac{1}{\ell(b_k)}
\log \big|(T_r^{\ell(b_k)})'(b_k)\big|
=
\int \log |T_r'| \, d\mu_1^r
=
2 \log \sqrt{r}
-
2 \log \left(\frac{r+1}{r}\right)^{-1}
\int_{-1/r}^0 \frac{\log(1-t)}{t}\, dt .
\]
 \end{theorem}

This theorem asserts that, for most sequences of $r-$rationals $p/q$ whose denominators tend to infinity, the corresponding $T_r-$orbits equidistribute with respect to the equilibrium measure associated to $-\log|T_r'|$. The statement does not hold for arbitrary sequences. For instance, if $r=1$ and $p/q = 1/q$, then $1/q = [q]$, so $\mu_{1/q} = \delta_{1/q}$ and hence
\[
\lim_{q \to \infty} \mu_{1/q} = \delta_0.
\]

Actually, the case $r=1$ has received a great deal of attention. Indeed, it is well known that if $p/q = [a_1,\dots,a_n]_1$ is a rational number in $(0,1)$ written in lowest terms, then the number of steps required by the Euclidean algorithm to compute $\gcd(p,q)$ is exactly $n$. Estimating the number of steps $\ell(p/q)$ needed in the Euclidean algorithm for typical pairs of integers was first studied by Heilbronn~\cite{he} and later by several other authors, for example \cite{bv, by, h, po}. It was shown that for most rationals $p/q$, the quantity $\ell(p/q)$ is asymptotically proportional to $\log q$ as $q \to \infty$. For example, Dixon~\cite{di} proved that there exists a constant $c_0>0$ such that, for any $\varepsilon>0$,
\[
\left| \ell(p,q) - \frac{12 \log 2}{\pi^2} \log q \right|
<
 (\log q)^{ \frac{1}{2}+ \varepsilon },
\]
for all but at most $x^2 \exp\!\big(c_0 (\log x)^{\varepsilon/2}\big)$ pairs of integers $1 \le p \le q \le x$. 

In 1969, Heilbronn~\cite{he} established a weaker averaged version of Theorem \ref{thm-cont-frac} for $r=1$. He proved that the sequence of measures
\[
\bar{\mu}_q
=
\frac{1}{\#\B_q}
\sum_{p \in \B_q} \mu_{p/q}
\]
converges, as $q \to \infty$, to the Gauss measure. More recently,  David and Shapira in~\cite[Theorem~1.1(2)]{ds} established a version of Theorem \ref{thm-cont-frac} again for the Gauss map (that is,  for $r=1$). The proof  in~\cite{ds} relies on methods from homogeneous dynamics. In contrast, in the present work, we have provided an alternative approach based on equidistribution results for periodic orbits that follow from large deviations estimates (Theorem~\ref{thm_equi2}). Interestingly, our results are flexible and can be applied to a wide range of examples, as the following proof shows.  

\begin{proof}[Proof of Theorem~\ref{thm:prob}]
We use the symbolic coding of the Gauss map. Let
\(\Sigma=\mathbb N^{\mathbb N}\) be the full shift and let
\(\pi:\Sigma\to(0,1)\) be the continued-fraction coding map. Define $\widehat\tau(x):=\log |G'(\pi x)|$ and
$\tau:=\frac12\left(\widehat\tau+\widehat\tau\circ\sigma\right)$. Then $\tau$ has summable variations and is bounded away from zero. Moreover, for every periodic point $x$ of period $m$, we have $S_m\tau(x)=S_m\widehat\tau(x)$.
Let $Y$ be the suspension over $(\Sigma,\sigma)$ with roof function $\tau$, and let $\Theta=(\Theta_t)_{t\ge0}$ be the corresponding suspension semi-flow. Recall that this suspension flow has topological entropy equal to $1$, and its measure of maximal entropy is $\nu_{mme}=AK(m_G),$
where $m_G$ is the Gauss measure in symbolic coordinates. Furthermore,
\[
\int \tau\,dm_G
=
\int \widehat\tau\,dm_G
=
\int \log |G'|\,d\mu_G
=
\frac{\pi^2}{6\log 2}.
\]
Set $\gamma:=\frac{6\log 2}{\pi^2}$. For a finite word $w=(a_1,\ldots,a_m)$, let $x_w=\overline{(a_1,\ldots,a_m)}$
be the corresponding periodic point, and $s(w):=S_m\tau(x_w).$
Let
\[
\mu_w:=\frac1m\sum_{j=0}^{m-1}\delta_{\sigma^j x_w}
\quad \text{ and } \quad \nu_w:=AK(\mu_w)
\]
be the corresponding invariant probability measure for the suspension flow. Then
\[
\int\tau\,d\mu_w=\frac{s(w)}{m},
\quad \text{ and therefore } \quad
\frac{m}{s(w)}
=
\left(\int\tau\,d\mu_w\right)^{-1}.
\]
Fix $\epsilon>0$. Define
\[
\mathcal E_\epsilon
:=
\left\{
\nu_w:
\left|
\frac{|w|}{s(w)}-\gamma
\right|
\ge \frac{\epsilon}{2}
\right\},
\]
and let
\[
\mathcal K_\epsilon
:=
\overline{\mathcal E_\epsilon}
\subset \mathcal M_{\le1}(\Theta),
\]
where the closure is taken in the cylinder topology. Thus
\(\mathcal K_\epsilon\) is closed by definition.
We claim that $\nu_{mme}\notin \mathcal K_\epsilon$.
Indeed, suppose by contradiction that there exists a sequence $(w_n)_n$ such that $\nu_{w_n}\to\nu_{mme}$
in the cylinder topology and for every $n \in \N$ we have
\[
\left|
\frac{|w_n|}{s(w_n)}-\gamma
\right|
\ge \frac{\epsilon}{2}
\]
Since both $\nu_{w_n}$ and $\nu_{mme}$ are probability measures, Remark~\ref{rem:convprob} implies that
$\mu_{w_n}\to m_G$ on cylinders and
\[
\int\tau\,d\mu_{w_n}
\longrightarrow
\int\tau\,dm_G .
\]
Consequently,
\[
\frac{|w_n|}{s(w_n)}
=
\left(\int\tau\,d\mu_{w_n}\right)^{-1}
\longrightarrow
\left(\int\tau\,dm_G\right)^{-1}
=
\gamma,
\]
which contradicts the definition of \(\mathcal E_\epsilon\). Hence
\(\nu_{mme}\notin\mathcal K_\epsilon\).

By Lemma~\ref{lem:beta}, or equivalently by the positivity of the rate
function away from the unique measure of maximal entropy, we have
\[
\beta_\epsilon:=\inf_{\nu\in\mathcal K_\epsilon} I_0(\nu)>0.
\]
Applying Theorem~\ref{thm_ldp2} with potential \(0\), and using the
full-shift/BIP version of the periodic-orbit estimates, there exist constants
\(C_\epsilon>0\), \(\rho_\epsilon>0\), and \(T_\epsilon>0\) such that, for all
\(T\ge T_\epsilon\),
\begin{equation}\label{eq:ldp-cumulative-prob}
\#\left\{
(x;s)\in\mathrm{FPer}(T):
\nu_x\in\mathcal K_\epsilon
\right\}
\le
C_\epsilon e^{(1-\rho_\epsilon)T}.
\end{equation}
Here \(\mathrm{FPer}(T)\) denotes periodic data of suspension length at most
\(T\). This cumulative estimate follows from the window estimate in
Theorem~\ref{thm_ldp2} by summing over intervals of fixed size; the resulting
polynomial factor is absorbed after decreasing \(\rho_\epsilon\), if
necessary.

We now compare suspension length with denominators of rationals. Let
\(p/r\in(0,1)\), \((p,r)=1\), and choose its canonical finite continued
fraction expansion
\[
\frac pr=[a_1,\ldots,a_m],
\]
so that \(m=\ell(p/r)\). Let \(w=(a_1,\ldots,a_m)\). By bounded distortion for
continued-fraction cylinders, there exists \(C_0>0\), independent of \(p/r\),
such that
\begin{equation}\label{eq:susp-denom-comparison}
\left|s(w)-2\log r\right|\le C_0 .
\end{equation}
Let $\tau_0:=\inf\tau>0.$ Since $s(w)\ge m\tau_0$, we obtain
\[
\begin{aligned}
\left|
\frac{m}{s(w)}
-
\frac{m}{2\log r}
\right|
&=
m
\left|
\frac{2\log r-s(w)}{s(w)\,2\log r}
\right|  
&\le
\frac{C_0m}{s(w)\,2\log r}
\le
\frac{C_0}{2\tau_0\log r}.
\end{aligned}
\]
Therefore, there exists \(R_\epsilon\ge1\) such that, whenever
\(r\ge R_\epsilon\),
\[
\left|
\frac{m}{2\log r}-\gamma
\right|
\ge \epsilon
\quad\Longrightarrow\quad
\left|
\frac{m}{s(w)}-\gamma
\right|
\ge \frac{\epsilon}{2}.
\]
Equivalently, for \(r\ge R_\epsilon\),
\[
\left|
\frac{\ell(p/r)}{2\log r}
-
\frac{6\log 2}{\pi^2}
\right|
\ge \epsilon
\quad\Longrightarrow\quad
\nu_w\in\mathcal K_\epsilon .
\]
Now let
\[
\Q(q)
=
\left\{
\frac pr\in(0,1):
1\le p<r\le q,\ (p,r)=1
\right\}
\]
and
\[
\Phi(q)=\#\Q(q)=\sum_{r\le q}\varphi(r).
\]
If \(p/r\in\Q(q)\), then \(r\le q\), and hence
\[
s(w)\le 2\log r+C_0\le 2\log q+C_0.
\]
Thus the set
\[
\left\{
\frac pr\in\Q(q):
\left|
\frac{\ell(p/r)}{2\log r}-\gamma
\right|
\ge\epsilon
\right\}
\]
is contained, up to the finitely many rationals with \(r<R_\epsilon\), in the
set of periodic data satisfying
\[
s(w)\le 2\log q+C_0
\qquad\text{and}\qquad
\nu_w\in\mathcal K_\epsilon .
\]
The finitely many exceptional rationals contribute \(O_\epsilon(1)\). Hence,
by \eqref{eq:ldp-cumulative-prob}, for all sufficiently large \(q\),
\[
\begin{aligned}
\#\left\{
\frac pr\in\Q(q):
\left|
\frac{\ell(p/r)}{2\log r}-\gamma
\right|
\ge\epsilon
\right\}
&\le
O_\epsilon(1)
+
C_\epsilon
\exp\left((1-\rho_\epsilon)(2\log q+C_0)\right) 
&\le
C'_\epsilon q^{2(1-\rho_\epsilon)} .
\end{aligned}
\]

Finally, since
\[
\Phi(q)=\sum_{r\le q}\varphi(r)\asymp q^2,
\]
there exists \(c>0\) such that \(\Phi(q)\ge cq^2\) for all sufficiently large
\(q\). Therefore
\[
\frac{1}{\Phi(q)}
\#\left\{
\frac pr\in\Q(q):
\left|
\frac{\ell(p/r)}{2\log r}-\gamma
\right|
\ge\epsilon
\right\}
\le
C''_\epsilon q^{-2\rho_\epsilon}.
\]
After decreasing the exponent once more, we obtain a constant
\(r_\epsilon>0\) such that, for all sufficiently large \(q\),
\[
\frac{1}{\Phi(q)}
\#
\left\{
\frac pr\in\Q(q):
\left|
\frac{\ell(p/r)}{2\log r}
-
\frac{6\log 2}{\pi^2}
\right|
\ge\epsilon
\right\}
\le
q^{-2r_\epsilon}.
\]
This proves the theorem.
\end{proof}

\end{document}